\documentclass[a4paper,12pt]{amsart}
\usepackage{graphicx}
\usepackage{amsmath}
\usepackage{latexsym,lscape}
\usepackage{amsthm,float }
\usepackage{latexsym,bezier,caption,amsbsy,psfrag,color,enumerate,amsfonts,amsmath,amscd,amssymb,comment}
\usepackage{epsfig}
\usepackage{rotating}
\usepackage{graphics}
\usepackage[latin1]{inputenc}

\newtheorem{definition}{Definition}[section]
\newtheorem{theorem}[definition]{Theorem}
\newtheorem{lemma}[definition]{Lemma}
\newtheorem{proposition}[definition]{Proposition}
\newtheorem{corollary}[definition]{Corollary}

\theoremstyle{definition}
\newtheorem{example}[definition]{Example}

\newtheorem{remark}[definition]{Remark}

\begin{document}


\title{On an uncountable family of  graphs whose spectrum is a Cantor set}
\author{Matteo Cavaleri}
\address{Matteo Cavaleri, Universit\`{a} degli Studi Niccol\`{o} Cusano - Via Don Carlo Gnocchi, 3 00166 Roma, Italia}
\email{matteo.cavaleri@unicusano.it}

\author{Daniele D'Angeli}
\address{Daniele D'Angeli, Universit\`{a} degli Studi Niccol\`{o} Cusano - Via Don Carlo Gnocchi, 3 00166 Roma, Italia}
\email{daniele.dangeli@unicusano.it}

\author{Alfredo Donno}
\address{Alfredo Donno, Universit\`{a} degli Studi Niccol\`{o} Cusano - Via Don Carlo Gnocchi, 3 00166 Roma, Italia}
\email{alfredo.donno@unicusano.it (Corresponding Author)}

\author{Emanuele Rodaro}
\address{Emanuele Rodaro, Politecnico di Milano - Piazza Leonardo da Vinci, 32 20133 Milano, Italia}
\email{emanuele.rodaro@polimi.it}

\begin{abstract}
For each $p\geq 1$, the star automaton group $\mathcal{G}_{S_p}$ is an automaton group which can be defined starting from a star graph on $p+1$ vertices. We study Schreier graphs associated with the action of the group $\mathcal{G}_{S_p}$ on the regular rooted tree $T_{p+1}$ of degree $p+1$ and on its boundary $\partial T_{p+1}$. With the transitive action on the $n$-th level of $T_{p+1}$ is associated a finite Schreier graph $\Gamma^p_n$, whereas there exist uncountably many orbits of the action on the boundary, represented by infinite Schreier graphs which are obtained as limits of the sequence $\{\Gamma_n^p\}_{n\geq 1}$ in the Gromov-Hausdorff topology. We obtain an explicit description of the spectrum of the graphs $\{\Gamma_n^p\}_{n\geq 1}$. Then, by using amenability of $\mathcal{G}_{S_p}$, we prove that the spectrum of each infinite Schreier graph is the union of a Cantor set of zero Lebesgue measure, which is the Julia set of the quadratic map $f_p(z) = z^2-2(p-1)z -2p$, and a countable collection of isolated points supporting the KNS spectral measure. We also give a complete classification of the infinite Schreier graphs up to isomorphism of unrooted graphs, showing that they may have $1$, $2$ or $2p$ ends, and that the case of $1$ end is generic with respect to the uniform measure on $\partial T_{p+1}$.
\end{abstract}

\maketitle

\begin{center}
\end{center}

\section{Introduction}
Schreier graphs are very popular in automaton group theory. In fact, they describe in a very natural way the action of an invertible automaton on words over an alphabet or, equivalently, on a regular rooted tree. This relates algebraic properties of the automaton group with combinatorial properties of the corresponding Schreier graphs. This paper can be framed into the exciting research field involving groups acting by automorphisms on rooted trees. Many papers have been devoted to these topics in the last decades: the interested reader can refer to the following list of works (and bibliography therein) for more details \cite{israel, dynamicssubgroup, fromtangled, GNS, nekrashevyc}. \\
Every automaton group acts by automorphisms on the rooted tree $T$. The action on finite levels is described by finite Schreier graphs. Going deeper and deeper in the tree, one is led to the study of the dynamical system $(G, \partial T, \nu)$ carrying the measure $\nu$  invariant under the action of the group $G$ on the boundary $\partial T$. One orbit of this action (i.e., a Schreier graph) can be seen as an infinite rooted graph obtained as limit of a sequence of finite rooted Schreier graphs in the Gromov-Hausdorff topology. Finite and infinite Schreier graphs have been investigated from a combinatorial point of view in several contexts (e.g., \cite{dimeri, tuttehanoi}). Classifications of infinite Schreier graphs have been studied in several papers (see \cite{intermediate, BDN, mio, basilica, perez} for further discussions about this topic).\\
In this setting, another problem is of considerable interest: the study of the spectral properties of Schreier graphs associated with an automaton group.
The determination of the spectrum of the Markov operator associated with a graph is, in general, a very difficult task, and only few examples are known for families of graphs. This analysis is very important in the theory of random walks on groups and in geometric group theory. It is remarkable that  the first examples of graphs whose spectrum is a Cantor set of Lebesgue measure zero, or the union of a Cantor set with a countable set of isolated points, have been obtained in the frame of Schreier graphs generated by automaton groups \cite{hecke}. In this context, the self-similar form of the generators reflects into the block structure of the adjacency matrix and in some special cases, an appropriate manipulation allows to find recursive formulae for the determination of the spectrum \cite{perez2, hanoi, marked, ihara}. This method produces the sequence of spectra corresponding to finite levels, and this sequence approximates the spectrum corresponding to the boundary action.  It is worth mentioning here that such approximation approach might also fail. In the case of the so called Basilica group, the situation seems to be more complicated and the renormalization of the infinite graph instead of the finite approximation is used (see \cite{bas} for more details).  \\
In the present work we want to study the two problems introduced above for the Schreier graphs associated with an infinite family of automaton groups.
More precisely this paper can be seen as a natural continuation of the paper  \cite{articolo0}, where we defined a particular class of automaton groups, called \emph{graph automaton groups}: starting from a graph $G=(V,E)$, we defined an invertible automaton $\mathcal{A}_G$ and then considered the associated group $\mathcal{G}_G$, whose generators are in a $1$-to-$1$ correspondence with $E$, and which acts by automorphisms on the regular rooted tree of degree $|V|$. The automaton $\mathcal{A}_G$ is bounded, so that the group $\mathcal{G}_G$ is amenable. Under the hypothesis $|E|\geq 2$, we showed that $\mathcal{G}_G$ is a fractal group which is weakly regular branch over its commutator subgroup $\mathcal{G}'_G$; moreover, it contains elements of finite order and has a number of torsion relators coming from directed cycles in $G$. It turns out that right angled Artin groups project onto the corresponding group obtained from the graph by this construction, which shows by the way that right angled Artin groups have amenable fractal weakly branch quotients. We also studied in \cite{articolo0} some properties of finite Schreier graphs associated with $\mathcal{G}_G$ when $G$ is a path graph or a cycle.

In the present paper, we consider a special class of graph automaton groups, obtained from a graph $G$ which is a star. We call such groups \emph{star automaton groups}. The star graph on $p+1$ vertices, consisting of a central vertex of degree $p$ and $p$ leaves, is denoted by $S_p$. Schreier graphs associated with the action of $\mathcal{G}_{S_p}$ on the regular rooted tree $T_{p+1}$ and its boundary $\partial T_{p+1}$ are the main object of research of this paper.\\
\indent In Section \ref{sectionpreliminaries} we recall the construction of graph automaton groups, together with the notion of finite and infinite Schreier graphs. We also recall some basic facts about the Ihara zeta function, both for a finite regular graph and for an infinite graph obtained as limit of a sequence of finite regular graphs; in particular, we focus on its integral representation by means of the KNS spectral measure.\\
\indent Section \ref{spectralsection} is devoted to spectral computations for finite and infinite Schreier graphs associated with the group $\mathcal{G}_{S_p}$. In Subsection \ref{secspectral3} all the details for the case $p=3$ are given. We construct the adjacency matrices of the finite Schreier graphs: by using the Schur complement technique, we find a recursive description of their characteristic polynomials in terms of a quadratic map in Theorem \ref{th13}. In Theorem \ref{propspettro}, the spectra of these matrices are explicitly described. Then, using amenability of the group $\mathcal{G}_{S_3}$, we prove in Theorem \ref{th23} that the spectrum of any infinite Schreier graph associated with $\mathcal{G}_{S_3}$ is the union of a Cantor set of zero Lebesgue measure, which is the Julia set of the quadratic map, and a countable collection of isolated points supporting the KNS spectral measure. The knowledge of the KNS spectral measure is then used to obtain an integral representation of the Ihara zeta function. The results obtained for the case $p=3$ are extended to the general case of any star graph $S_p$, and they are presented in Theorem \ref{th1p}, Theorem \ref{propspettrop}, and Theorem \ref{thmspettrop} of Subsection \ref{secspectralp}.\\
\indent Section \ref{sectionschreier} is devoted to the investigation of topological and isomorphism properties of Schreier graphs associated with $\mathcal{G}_{S_p}$. The topological investigation developed for the finite case in Subsection \ref{sectionfinite} is preliminary to the results obtained in Subsection \ref{sectioninfinite} in the infinite case, where we are able to classify, up to isomorphism of unrooted graphs, all infinite orbital Schreier graphs. We show that the limit graphs may have $1$, $2$ or $2p$ ends. In Theorem \ref{fini} we give an explicit classification of infinite Schreier graphs of $\mathcal{G}_{S_p}$ in terms of infinite words in $\{0,1,\ldots, p\}$, by characterizing the elements of the boundary of the tree $T_{p+1}$ belonging to a graph with $1$, $2$, or $2p$ ends, showing that there exist uncountably many $1$-ended and $2$-ended orbits, but exactly one $2p$-ended orbit. Moreover, the case of $1$ end is generic with respect to the uniform measure on $\partial T_{p+1}$. In Theorem \ref{teo_iso}, we provide necessary and sufficient conditions for two elements of $\partial T_{p+1}$ to belong to isomorphic infinite Schreier graphs. In particular, we prove that there exists one isomorphism class of $2p$-ended graphs, consisting of one orbit; there exist uncountably many isomorphism classes of $2$-ended graphs, each consisting of $2p$ graphs; there exist uncountably many isomorphism classes of $1$-ended graphs, each consisting of uncountably many graphs. Finally, each isomorphism class is proven to have zero measure in Corollary \ref{finalcorozero}.

\section{Preliminaries}\label{sectionpreliminaries}
In this preliminary section, we recall some basic definitions and properties about automaton groups and their Schreier graphs, focusing on the special class of automaton groups, called graph automaton groups, which has been introduced by the authors in \cite{articolo0}. We also recall the notion of KNS spectral measure and Ihara zeta function, which will be investigated in Section \ref{spectralsection} in the case of star automaton groups.

\subsection{Graph automaton groups and Schreier graphs}
Let us start be recalling the basic definition of automaton.
\begin{definition}
\indent An \textit{automaton} is a quadruple $\mathcal{A} =
(S,X,\lambda,\eta)$, where:
\begin{enumerate}
\item $S$ is the set of states;
\item $X=\{1,2,\ldots, k\}$ is an alphabet;
\item $\lambda: S\times X \rightarrow S$ is the restriction map;
\item $\eta: S\times X \rightarrow X$ is the output map.
\end{enumerate}
\end{definition}
The automaton $\mathcal{A}$ is \textit{finite} if $S$ is finite and it is \textit{invertible} if, for all $s\in
S$, the transformation $\eta(s, \cdot):X\rightarrow X$ is a permutation of $X$. An automaton $\mathcal{A}$ can be visually
represented by its \textit{Moore diagram}: this is a directed labeled graph whose vertices are identified with the states of
$\mathcal{A}$. For every state $s\in S$ and every letter $x\in X$, the diagram has an arrow from $s$ to $\lambda(s,x)$
labeled by $x|\eta(s,x)$. A sink $id$ in $\mathcal{A}$ is a state with the property that $\lambda(id,x)=id$ and $\eta(id,x)=x$ for any $x\in X$.        \\
\indent An important class of automata is given by bounded automata \cite{sidki}. An automaton is said to be \emph{bounded} if the sequence of numbers of paths of length $n$ avoiding the sink state (along the directed edges of the Moore diagram) is bounded.

For each $n\geq 1$, let $X^n$ denote the set of words of length $n$ over the alphabet $X$ and put $X^0  = \{\emptyset\}$, where $\emptyset$ is the empty word. Then the action of $\mathcal{A}$ can be naturally extended to the infinite set $X^\ast= \bigcup_{n=0}^\infty X^n$ and to the set $X^\infty = \{x_1x_2x_3\ldots : x_i\in X\}$ of infinite words over $X$.\\
\indent For a state $s\in S$, we denote by $\mathcal{A}_s$ the transformation $\eta(s,\cdot)$ of $X^{\ast}\cup X^\infty$. Given the invertible automaton $\mathcal{A}$, the \textit{automaton group} generated by $\mathcal{A}$ is by definition the group generated by the transformations $\mathcal{A}_s$, for $s\in S$, and it is denoted $G(\mathcal{A})$. In the rest of the paper, we will often use the notation $s$ instead of $\mathcal{A}_s$. Notice that the action of $G(\mathcal{A})$ on $X^\ast$ preserves the sets $X^n$, for each $n$.\\
\indent It is a remarkable fact that an automaton group can be regarded in a very natural way as a group of automorphisms of the regular rooted tree $T_k$ in which each vertex has $|X|=k$ children, via the identification of the $k^n$ vertices of the $n$-th level of $T_k$ with the set $X^n$. Similarly, the action on $X^\infty$ can be regarded as an action on the boundary $\partial T_k$ of the tree, whose elements are infinite geodesic rays starting at the root of $T_k$. Notice that the set $X^\infty$ can be equipped with the direct product topology; it is totally disconnected and homeomorphic to the Cantor set. We will denote by $\nu$ the uniform measure on $X^\infty$ or, equivalently, on $\partial T_k$.

\indent The group $G(\mathcal{A})$  is said to be \textit{spherically transitive} if its action is transitive on $X^n$, for any $n$. Let $g\in G(\mathcal{A})$. The action of $g$ on $X^\ast$ can be factorized by considering the action on $X$ and $|X|$ restrictions as follows. Let $Sym(k)$ be the symmetric group on $k$ elements. Then an element $g\in G(\mathcal{A})$ can be represented as
\begin{eqnarray}\label{ssd}
g=(g_1,\ldots, g_{k})\sigma,
\end{eqnarray}
where $g_i:=\lambda(g,i)\in G(\mathcal{A})$ and $\sigma\in Sym(k)$ describes the action of $g$ on $X$. We say that Eq. \eqref{ssd} is the \emph{self-similar representation} of $g$. In the tree interpretation of Eq. \eqref{ssd}, the permutation $\sigma$ corresponds to the action of $g$ on the first level of $T_k$, and the automorphism $g_i$ is the restriction of the action of $g$ to the subtree (isomorphic to the whole $T_k$) rooted at the $i$-th vertex of the first level.
Finally, it is known that if the automaton $\mathcal{A}$ is bounded, then the group $G(\mathcal{A})$ is amenable (see, e.g., \cite{amenability bounded}).

In \cite{articolo0} we introduced the following construction associating an invertible automaton with a given finite graph.\\
\indent Let $G=(V,E)$ be a finite graph, where $V=\{x_1,\ldots, x_k\}$ is its vertex set and $E$ is its edge set. Let $E'$ be the set of edges, where an orientation of each edge has been chosen. Notice that elements in $E$ are unordered pairs of type $\{x_i,x_j\}$, whereas elements in $E'$ are ordered pairs of type $(x_i,x_j)$, meaning that the edge has been oriented from the vertex $x_i$ to the vertex $x_j$.\\
\indent We then define an automaton $\mathcal{A}_G=(E' \cup \{id\}, V, \lambda, \eta)$ such that:
\begin{itemize}
\item $E' \cup \{id\}$ is the set of states;
\item $V$ is the alphabet;
\item $\lambda: E'\times V\to E'$ is the restriction map such that, for each $e=(x,y)\in E'$, one has
$$
\lambda (e,z) = \left\{
                  \begin{array}{ll}
                    e & \hbox{if } z=x \\
                    id & \hbox{if } z\neq x;
                  \end{array}
                \right.
$$
\item $\eta: E'\times V\to V$ is the output map such that, for each $e=(x,y)\in E'$, one has
$$
\eta (e,z) = \left\{
                  \begin{array}{ll}
                    y & \hbox{if } z=x \\
                    x & \hbox{if } z=y \\
                    z & \hbox{if } z\neq x,y.
                  \end{array}
                \right.
$$
\end{itemize}
In other words, any directed edge $e=(x,y)$ is a state of the automaton $\mathcal{A}_G$ and it has just one restriction to itself (given by $\lambda(e,x)$) and all other restrictions to the sink $id$. Its action is nontrivial only on the letters $x$ and $y$, which are switched since $\eta(e,x)=y$ and $\eta(e,y)=x$. It is easy to check that
$\mathcal{A}_G$ is invertible for any $G$ and any choice of the orientation of the edges. The \emph{graph automaton group} $\mathcal{G}_G$ is  defined as the automaton group generated by $\mathcal{A}_G$. In \cite[Theorem 3.7]{articolo0} it is shown that, whenever $|E|\geq 2$, the automaton $\mathcal{A}_G$ is bounded, so that the group $\mathcal{G}_G$ is amenable; moreover, $\mathcal{G}_G$ is a fractal group and it is weakly regular branch over its commutator subgroup $\mathcal{G}'_G$.

For any integer $p\geq 1$, let $S_p = (V_p, E_p)$ denote the \emph{star graph} on $p+1$ vertices. Let us identify its vertex set $V_p$ with the set $\{0,1,2,\ldots, p\}$, where $0$ corresponds to the central vertex, which is the only vertex of degree $p$, and the $p$ leaves are identified with the vertex subset $\{1,2,\ldots, p\}$ (see Fig. \ref{vaccino} for the case $p=6$).

\begin{figure}[h]
\begin{center}
\psfrag{0}{$0$}\psfrag{1}{$1$}\psfrag{2}{$2$}\psfrag{3}{$3$}\psfrag{4}{$4$}\psfrag{5}{$5$}\psfrag{6}{$6$}
\includegraphics[width=0.3\textwidth]{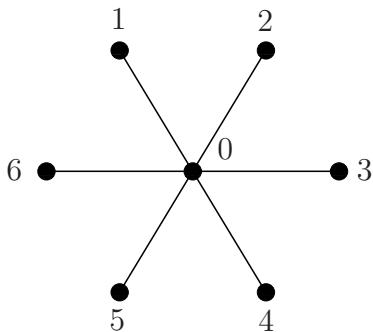}
\end{center}\caption{The star graph $S_6$.} \label{vaccino}
\end{figure}

In this paper we will deal with \emph{star automaton groups}, which are automaton groups obtained from $S_p$ following the construction described above. The star automaton group defined starting from the graph $S_p$ will be denoted $\mathcal{G}_{S_p}$.

We conclude this subsection by recalling the definition of finite and infinite Schreier graphs associated with an automaton group $G(\mathcal{A})$.

\begin{definition}
The $n$-th Schreier graph $\Gamma_n=(V_{\Gamma_n}, E_{\Gamma_n})$ of the action of $G(\mathcal{A})$ on $T_k$, with respect to a symmetric
generating set $S$, is the graph whose vertex set is $X^n$, where two vertices $u$ and $v$ are adjacent if and only if there exists $s\in S$ such that $s(u) = v$. If this is the case, the edge from $u$ to $v$ is labeled by $s$.
\end{definition}
Notice that the Schreier graph $\Gamma_n$ is a regular graph of degree $|S|$ on $k^n$ vertices and it
is connected for each $n$ under the hypothesis of spherical transitivity. For each $n\geq 1$, let $\pi_{n+1}:\Gamma_{n+1}\longrightarrow \Gamma_n$ be the map defined on $V_{\Gamma_{n+1}}$ as
$$
\pi_{n+1}(x_1\ldots x_nx_{n+1}) = x_1\ldots x_n.
$$
This map induces a surjective morphism from $\Gamma_{n+1}$ onto $\Gamma_n$, which is a graph covering of degree $k$.
In the rest of the paper, we will denote by $A_n$ the adjacency matrix of the Schreier graph $\Gamma_n$: by definition, this is a symmetric square matrix of size $k^n$ whose rows (and columns) sum to $|S|$. Since $A_n$ is symmetric, all its eigenvalues are real: they constitute the adjacency spectrum (or spectrum) of $\Gamma_n$. Notice that the normalized adjacency matrix of $\Gamma_n$, which is given by $\frac{1}{|S|}A_n$, can be regarded as the transition matrix of the Markov operator $M_n$ associated with the simple random walk on $\Gamma_n$.

For each $n\geq 1$, the Schreier graph $\Gamma_n$ is nothing but the orbital graph of the action of $G(\mathcal{A})$ on the $n$-th level of the tree $T_k$ or, equivalently, on the set $X^n$.
On the other hand, it also makes sense to consider orbital graphs associated with the action of $G(\mathcal{A})$ on $\partial T_k$ or, equivalently, on the set $X^\infty$. Since the action of $G(\mathcal{A})$ on $X^\infty$ has uncountably many orbits, there exist uncountably many distinct infinite Schreier graphs which are possibly nonisomorphic.\\
\indent Now take an infinite word $w=x_1x_2x_3\ldots \in X^\infty$, and denote by $w_n= x_1\ldots x_n\in X^n$ its prefix of length $n$.
It is known that the infinite Schreier graph $\Gamma_w$ describing the orbit of $w$ is approximated, as a rooted graph $(\Gamma_w, w)$, by the sequence of
finite Schreier graphs $(\Gamma_n, w_n)$, in the space of rooted graphs of uniformly bounded degree endowed with the Gromov-Hausdorff convergence, provided, for example, by the following metric: given two rooted graphs $(\Gamma_1, v_1)$ and $(\Gamma_2, v_2)$, one put
$$
dist((\Gamma_1, v_1),(\Gamma_2, v_2)) = \inf\left\{\frac{1}{r+1} : B_{\Gamma_1}(v_1,r) \textrm{ is isomorphic to } B_{\Gamma_2}(v_2,r)\right\},
$$
where $B_{\Gamma_i}(v_i,r)$ is the ball of radius $r$ in $\Gamma_i$ centered in $v_i$ (see Theorem 3 in \cite{marked}).

According to the theory developed, for instance, in \cite{hecke}, under the hypothesis of amenability of the group $G(\mathcal{A})$, the spectrum of any infinite orbital Schreier graph $\Gamma$ is obtained as
\begin{eqnarray*}
\textrm{spectrum}(\Gamma) = \overline{\bigcup_{n=0}^\infty \textrm{spectrum}(\Gamma_n)}.
\end{eqnarray*}

\subsection{Ihara zeta function}
In this section we recall the definition of Ihara zeta function for a finite regular graph $\Gamma$, which is an analogue of the Riemann's zeta function. For more details, the reader is referred to \cite{ihara}.

\begin{definition}
The Ihara zeta function $\zeta_\Gamma(t)$ for a finite regular graph $\Gamma$ is the function
$$
\zeta_\Gamma(t) = \exp\left(\sum_{r=1}^\infty \frac{c_rt^r}{r}\right),
$$
where $c_r$ is the number of closed, oriented loops of length $r$ in the graph $\Gamma$.
\end{definition}

It is also known that the Ihara zeta function of a finite regular graph $\Gamma= (V_\Gamma, E_\Gamma)$ of degree $k$ satisfies the equation
$$
\zeta_\Gamma(t) = (1-t^2)^{-\frac{k-2}{2}|V_\Gamma|} \det (1-tkM +(k-1)t^2)^{-1},
$$
where $M$ is the Markov operator on $\Gamma$.

A notion of Ihara zeta function for an infinite rooted graph which is the limit of a sequence of finite regular rooted graphs can be given.
Let $(\Gamma_n,v_n)$ be a sequence of finite rooted graphs regular of degree $k$ converging to the limit graph $(\Gamma,v)$, and let $M_n$ be the Markov operator on $\Gamma_n$ whose transition matrix is the normalized adjacency matrix of $\Gamma_n$.
The eigenvalues $\lambda_{i,n}$ of the operator $M_n$ are said to be equidistributed with respect to a measure $\mu$ which has support in $[-1,1]$
if the sequence of counting measures
\begin{eqnarray}\label{mcountingm}
\mu_n= \sum_{i=1}^{|V_{\Gamma_n}|} \frac{\delta_{\lambda_{i,n}}}{|V_{\Gamma_n}|}
\end{eqnarray}
weakly converges to the measure $\mu$. Moreover, it is known that given a covering sequence $(\Gamma_n,v_n)$ of finite $k$-regular graphs, with associated Markov operators $M_n$, the eigenvalues of $M_n$ are equidistributed with respect to some measure $\mu$, which is called the Kesten-Neumann-Serre (KNS) spectral measure of the limit graph $\Gamma$. In particular
\begin{eqnarray*}
\frac{1}{|V_{\Gamma_n}|}\ln\zeta_{\Gamma_n}(t) &=&  \sum_{r=1}^\infty \frac{c_r(\Gamma_n)t^r}{|V_{\Gamma_n}|r}\\
&=& -\frac{k-2}{2} \ln(1-t^2)-\frac{1}{|V_{\Gamma_n}|} \ln\det (1-tkM_n +(k-1)t^2).
\end{eqnarray*}
When $n$ goes to $\infty$, one gets:
\begin{eqnarray*}
\ln\zeta_\Gamma(t)= \lim_{n\to \infty} \frac{1}{|V_{\Gamma_n}|}\ln\zeta_{\Gamma_n}(t)  = \sum_{r=1}^\infty \frac{\widetilde{c_r}t^r}{r},
\end{eqnarray*}
where $\widetilde{c_r}$ is the limit of the sequence $\frac{c_r(\Gamma_n)}{|V_{\Gamma_n}|}$. Moreover, the KNS spectral measure is uniquely determined by the Ihara zeta function $\zeta_\Gamma(t)$ according to the equation
$$
\ln \zeta_\Gamma(t) = -\frac{k-2}{2}\ln(1-t^2)-\int_{-1}^1 \ln(1-tk\lambda +(k-1)t^2) d\mu(\lambda), \quad \forall \ t : |t|<\frac{1}{k-1}.
$$
We will apply this machinery in the setting of infinite orbital Schreier graphs, obtained as limits of sequences of finite Schreier graphs, for the star automaton group $\mathcal{G}_{S_p}$.

\section{Spectrum of Schreier graphs of the star automaton group $\mathcal{G}_{S_p}$} \label{spectralsection}

This section is devoted to the computation of the spectrum of both finite and infinite Schreier graphs associated with the action of the star automaton group $\mathcal{G}_{S_p}$ on the set $X^\ast \cup X^\infty$, where $X=\{0,1,\ldots, p\}$, or equivalently, on the regular rooted tree $T_{p+1}$ and on its boundary. Since the same argument holds for every $p$, we prefer to present the explicit computation for the case $p=3$ for the convenience of the reader; then we will extend the claim to the general case.

\subsection{The case $p=3$}\label{secspectral3}

Consider the oriented star graph $S_3$ on the four vertices $\{0,1,2,3\}$ depicted in Fig. \ref{figureS3}.
\begin{figure}[h]
\begin{center}
\psfrag{a}{$a$}\psfrag{b}{$b$}\psfrag{c}{$c$}\psfrag{0}{$0$}\psfrag{1}{$1$}\psfrag{2}{$2$}\psfrag{3}{$3$}
\includegraphics[width=0.25\textwidth]{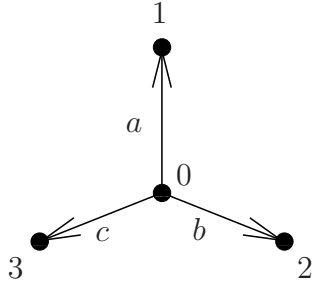}
\end{center}\caption{The oriented star graph $S_3$.} \label{figureS3}
\end{figure}

The automaton associated with such orientation of $S_3$ is given in Fig. \ref{automatonS3}.
\begin{figure}[h]
\begin{center}
\footnotesize
\psfrag{a}{$a$}\psfrag{b}{$b$}\psfrag{c}{$c$}\psfrag{id}{$id$}\psfrag{ia}{$1|0, 2|2, 3|3$}\psfrag{ib}{$2|0, 1|1, 3|3$}\psfrag{ic}{$3|0, 1|1, 2|2$}\psfrag{ta}{$0|1$}\psfrag{tb}{$0|2$}\psfrag{tc}{$0|3$}
\includegraphics[width=0.40\textwidth]{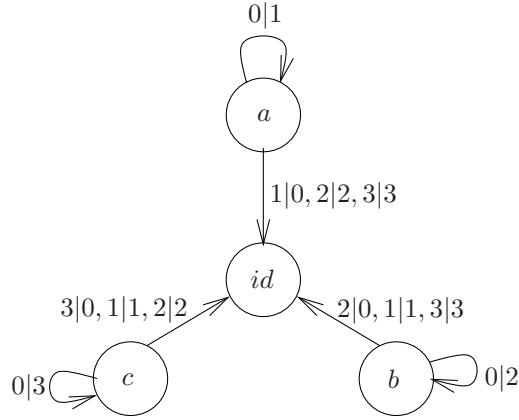}
\end{center}\caption{The automaton associated with the graph $S_3$ of Fig. \ref{figureS3}.} \label{automatonS3}
\end{figure}

In particular, the star automaton group $\mathcal{G}_{S_3}$ is the group generated by the three automorphisms having the following self-similar representation (see \cite{articolo0}):
\begin{eqnarray}\label{generators3}
a=(a, id, id, id)(01) \qquad b=(b,id,id,id)(02) \qquad c=(c,id,id,id)(03).
\end{eqnarray}
Moreover, one has:
$$
a^{-1}=(id, a^{-1}, id, id)(01) \qquad b^{-1}=(id,id,b^{-1},id)(02) \qquad c^{-1}=(id,id,id,c^{-1})(03).
$$
Let us denote by $a_n, b_n,c_n$ the permutation matrices of size $4^n$ describing the action of the automorphisms $a,b,c$, respectively, on the set $\{0,1,2,3\}^n$, so that the adjacency matrix $A_n$ of the $n$-th Schreier graph $\Gamma_n$ is given by
$$
A_n = a_n+ a_n^{-1}+ b_n+ b_n^{-1}+c_n+ c_n^{-1}.
$$
From Eq. \eqref{generators3} we get:
$$
a_{n+1}= \left(
           \begin{array}{c|c|c|c}
             0 & a_n & 0 & 0 \\ \hline
             I_n & 0 & 0 & 0 \\ \hline
             0 & 0 & I_n & 0 \\ \hline
             0 & 0 & 0 & I_n \\
           \end{array}
         \right)            \quad  b_{n+1}= \left(
           \begin{array}{c|c|c|c}
             0 & 0 & b_n & 0 \\ \hline
             0 & I_n & 0 & 0 \\ \hline
             I_n & 0 & 0 & 0 \\ \hline
             0 & 0 & 0 & I_n \\
           \end{array}
         \right)          \quad  c_{n+1}= \left(
           \begin{array}{c|c|c|c}
             0 & 0 & 0 & c_n \\ \hline
             0 & I_n & 0 & 0 \\ \hline
             0 & 0 & I_n & 0 \\ \hline
             I_n & 0 & 0 & 0 \\
           \end{array}
         \right)
$$
where $I_n$ is the identity matrix of size $4^n$ and $0$ is the zero matrix of size $4^n$. Similarly:
$$
a^{-1}_{n+1}= \left(
           \begin{array}{c|c|c|c}
             0 & I_n & 0 & 0 \\ \hline
             a_n^{-1} & 0 & 0 & 0 \\ \hline
             0 & 0 & I_n & 0 \\ \hline
             0 & 0 & 0 & I_n \\
           \end{array}
         \right)            \   b_{n+1}^{-1}= \left(
           \begin{array}{c|c|c|c}
             0 & 0 & I_n & 0 \\ \hline
             0 & I_n & 0 & 0 \\ \hline
             b_n^{-1} & 0 & 0 & 0 \\ \hline
             0 & 0 & 0 & I_n \\
           \end{array}
         \right)          \   c_{n+1}^{-1}= \left(
           \begin{array}{c|c|c|c}
             0 & 0 & 0 & I_n \\ \hline
             0 & I_n & 0 & 0 \\ \hline
             0 & 0 & I_n & 0 \\ \hline
             c_n^{-1} & 0 & 0 & 0 \\
           \end{array}
         \right).
$$
Notice that
$$
a_1=a_1^{-1}= \left(
           \begin{array}{cccc}
             0 & 1 & 0 & 0 \\
             1 & 0 & 0 & 0 \\
             0 & 0 & 1 & 0 \\
             0 & 0 & 0 & 1 \\
           \end{array}
         \right)         \  b_1=b_1^{-1}= \left(
           \begin{array}{cccc}
             0 & 0 & 1 & 0 \\
             0 & 1 & 0 & 0 \\
             1 & 0 & 0 & 0 \\
             0 & 0 & 0 & 1 \\
           \end{array}
         \right)        \  c_1=c_1^{-1}= \left(
           \begin{array}{cccc}
             0 & 0 & 0 & 1 \\
             0 & 1 & 0 & 0 \\
             0 & 0 & 1 & 0 \\
             1 & 0 & 0 & 0 \\
           \end{array}
         \right).
$$
Therefore, the adjacency matrix of the Schreier graph $\Gamma_{n+1}$ is
\begin{eqnarray*}
A_{n+1} = \left(
           \begin{array}{c|c|c|c}
             0 & a_n+I_n & b_n+I_n & c_n+I_n \\ \hline
             a_n^{-1}+I_n & 4I_n & 0 & 0 \\ \hline
             b_n^{-1}+I_n & 0 & 4I_n & 0 \\ \hline
             c_n^{-1}+I_n & 0 & 0 & 4I_n \\
           \end{array}
         \right),  \mbox{with }  A_1 = \left(
           \begin{array}{cccc}
             0 & 2 & 2 & 2 \\
             2 & 4  & 0 & 0 \\
             2 & 0 & 4  & 0 \\
             2 & 0 & 0 & 4  \\
           \end{array}
         \right).
\end{eqnarray*}

We will make use of the following well known result about determinant computation via the Schur complement formula (see, for instance, \cite{schur}).
\begin{lemma}\label{lemmaschur}
Let $M= \left(
          \begin{array}{c|c}
            A & B \\ \hline
            C & D \\
          \end{array}
        \right)$ be a block matrix, where $A$ has size $k\times k$, $B$ has size $k\times (n-k)$, $C$ has size $(n-k)\times k$, and $D$ has size $(n-k)\times (n-k)$. If $D$ is nonsingular, one has
$$
\det M = \det D \cdot \det (A - BD^{-1}C),
$$
where the matrix $M/D := A - BD^{-1}C$ is called the Schur complement of $D$.
\end{lemma}

\begin{theorem}\label{th13}
Let $P_n(\lambda)$ be the characteristic polynomial of the adjacency matrix $A_n$ of the  Schreier graph $\Gamma_n$, for each $n\geq 1$. Then
\begin{eqnarray}\label{ricorrenza}
P_{n+1}(\lambda) = (\lambda-4)^{2\cdot 4^n} P_n(f(\lambda)),
\end{eqnarray}
with $f(\lambda) = \lambda^2-4\lambda -6$ and $P_1(\lambda) = (\lambda-6)(\lambda+2)(\lambda-4)^2$.
\end{theorem}

\begin{proof}
A direct computation gives $P_1(\lambda)=(\lambda-6)(\lambda+2)(\lambda-4)^2$.\\
Now put $A'_{n+1}=A_{n+1}-\lambda I_{n+1}$, so that
\begin{eqnarray*}
P_{n+1}(\lambda) = \det A'_{n+1} = \det \left(
           \begin{array}{c|c|c|c}
             -\lambda I_n & a_n+I_n & b_n+I_n & c_n+I_n \\ \hline
             a_n^{-1}+I_n & (4-\lambda)I_n & 0 & 0 \\ \hline
             b_n^{-1}+I_n & 0 & (4-\lambda)I_n & 0 \\ \hline
             c_n^{-1}+I_n & 0 & 0 & (4-\lambda)I_n \\
           \end{array}
         \right).
\end{eqnarray*}
In order to compute $\det A'_{n+1}$, we use the Schur complement technique, where
$$
A=-\lambda I_n;    \quad B = \left(\begin{array}{c|c|c}
                               a_n+I_n & b_n+I_n & c_n+I_n
                             \end{array} \right);
\quad
C= \left(
                                                              \begin{array}{c}
                                                                a_n^{-1}+I_n \\
                                                                b_n^{-1}+I_n \\
                                                                c_n^{-1}+I_n \\
                                                              \end{array}
                                                            \right); \quad D=  (4-\lambda)I_{3n}.
$$
The Schur complement of the block $D=(4-\lambda)I_{3n}$ is given by
\begin{eqnarray*}
A'_{n+1}/D&=&
-\lambda I_n - \left(
                 \begin{array}{ccc}
                   a_n+I_n & b_n+I_n & c_n+I_n \\
                 \end{array}
               \right) \cdot \frac{1}{4-\lambda}I_{3n}\cdot \left(
                                                              \begin{array}{c}
                                                                a_n^{-1}+I_n \\
                                                                b_n^{-1}+I_n \\
                                                                c_n^{-1}+I_n \\
                                                              \end{array}
                                                            \right)\\
&=& -\lambda I_n -\frac{1}{4-\lambda} (a_n + a_n^{-1}+2I_n +b_n + b_n^{-1}+2I_n+c_n + c_n^{-1}+2I_n)\\
&=& \left(-\lambda-\frac{6}{4-\lambda}\right)I_n -\frac{1}{4-\lambda}A_n = \frac{\lambda^2-4\lambda-6}{4-\lambda}I_n - \frac{1}{4-\lambda}A_n.
\end{eqnarray*}
Therefore, we have
\begin{eqnarray*}
\det A_{n+1}' &=& \det D \cdot \det (A'_{n+1}/D)  \\
&=& (4-\lambda)^{3\cdot 4^n} \cdot \det \left(\frac{\lambda^2-4\lambda-6}{4-\lambda}I_n - \frac{1}{4-\lambda}A_n\right)\\
&=& (4-\lambda)^{2\cdot 4^n} \cdot \det (A_n -(\lambda^2-4\lambda-6)I_n).
\end{eqnarray*}
This completes the proof.
\end{proof}

\begin{remark}\rm
Observe that, if we define $\Gamma_0$ to be the graph consisting of a single vertex endowed with three loops, so that it is a regular graph of degree $6$ as the graph $\Gamma_n$ is for each $n\geq 1$, then we have $P_0(\lambda)= \lambda -6$ and Eq. \eqref{ricorrenza} still holds with $n=0$, since
$$
P_1(\lambda) = (\lambda-4)^2 \cdot P_0(f(\lambda)),
$$
because $P_0(f(\lambda)) = \lambda^2-4\lambda-12 = (\lambda-6)(\lambda+2)$.
\end{remark}

In Fig. \ref{alf1} and Fig. \ref{alf2} the Schreier graphs $\Gamma_n$, for $n=1,2,3$, associated with the group $\mathcal{G}_{S_3}$ are depicted. Vertices are labeled by words in $\{0,1,2,3\}^n$.

\begin{figure}[h]
\begin{center}
\psfrag{1}{$1$}\psfrag{2}{$2$}\psfrag{3}{$3$}\psfrag{0}{$0$}
\psfrag{00}{$00$}\psfrag{01}{$01$}\psfrag{02}{$02$}\psfrag{03}{$03$}
\psfrag{10}{$10$}\psfrag{11}{$11$}\psfrag{12}{$12$}\psfrag{13}{$13$}
\psfrag{20}{$20$}\psfrag{21}{$21$}\psfrag{22}{$22$}\psfrag{23}{$23$}
\psfrag{30}{$30$}\psfrag{31}{$31$}\psfrag{32}{$32$}\psfrag{33}{$33$}
\psfrag{C12}{$C^1_2$}
\includegraphics[width=0.85\textwidth]{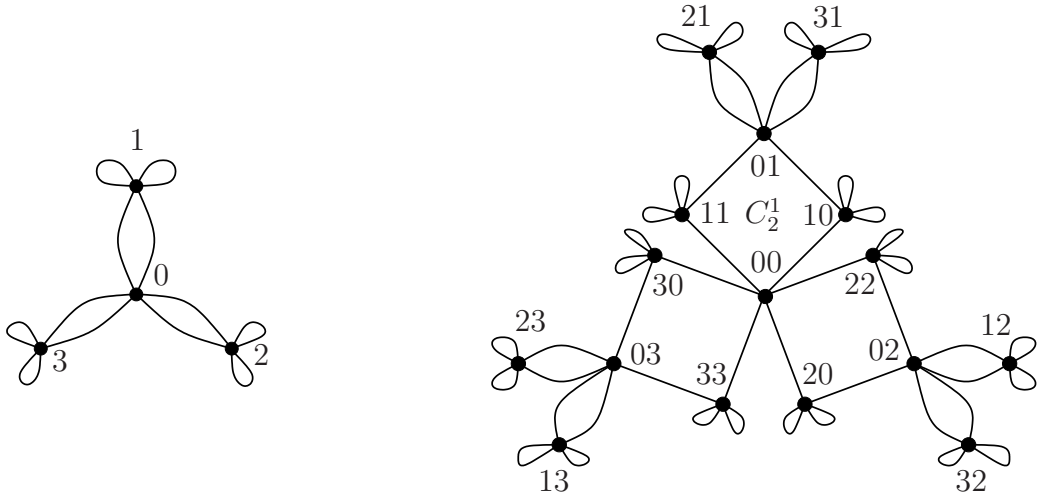}
\end{center}\caption{The Schreier graphs $\Gamma_1$ and $\Gamma_2$ associated with $\mathcal{G}_{S_3}$.} \label{alf1}
\end{figure}

\begin{figure}[h]
\begin{center}    \footnotesize
\psfrag{000}{$000$}\psfrag{001}{$001$}\psfrag{002}{$002$}\psfrag{003}{$003$}
\psfrag{010}{$010$}\psfrag{011}{$011$}\psfrag{012}{$012$}\psfrag{013}{$013$}
\psfrag{020}{$020$}\psfrag{021}{$021$}\psfrag{022}{$022$}\psfrag{023}{$023$}
\psfrag{030}{$030$}\psfrag{031}{$031$}\psfrag{032}{$032$}\psfrag{033}{$033$}

\psfrag{100}{$100$}\psfrag{101}{$101$}\psfrag{102}{$102$}\psfrag{103}{$103$}
\psfrag{110}{$110$}\psfrag{111}{$111$}\psfrag{112}{$112$}\psfrag{113}{$113$}
\psfrag{120}{$120$}\psfrag{121}{$121$}\psfrag{122}{$122$}\psfrag{123}{$123$}
\psfrag{130}{$130$}\psfrag{131}{$131$}\psfrag{132}{$132$}\psfrag{133}{$133$}

\psfrag{200}{$200$}\psfrag{201}{$201$}\psfrag{202}{$202$}\psfrag{203}{$203$}
\psfrag{210}{$210$}\psfrag{211}{$211$}\psfrag{212}{$212$}\psfrag{213}{$213$}
\psfrag{220}{$220$}\psfrag{221}{$221$}\psfrag{222}{$222$}\psfrag{223}{$223$}
\psfrag{230}{$230$}\psfrag{231}{$231$}\psfrag{232}{$232$}\psfrag{233}{$233$}

\psfrag{300}{$300$}\psfrag{301}{$301$}\psfrag{302}{$302$}\psfrag{303}{$303$}
\psfrag{310}{$310$}\psfrag{311}{$311$}\psfrag{312}{$312$}\psfrag{313}{$313$}
\psfrag{320}{$320$}\psfrag{321}{$321$}\psfrag{322}{$322$}\psfrag{323}{$323$}
\psfrag{330}{$330$}\psfrag{331}{$331$}\psfrag{332}{$332$}\psfrag{333}{$333$}
 \includegraphics[width=0.85\textwidth]{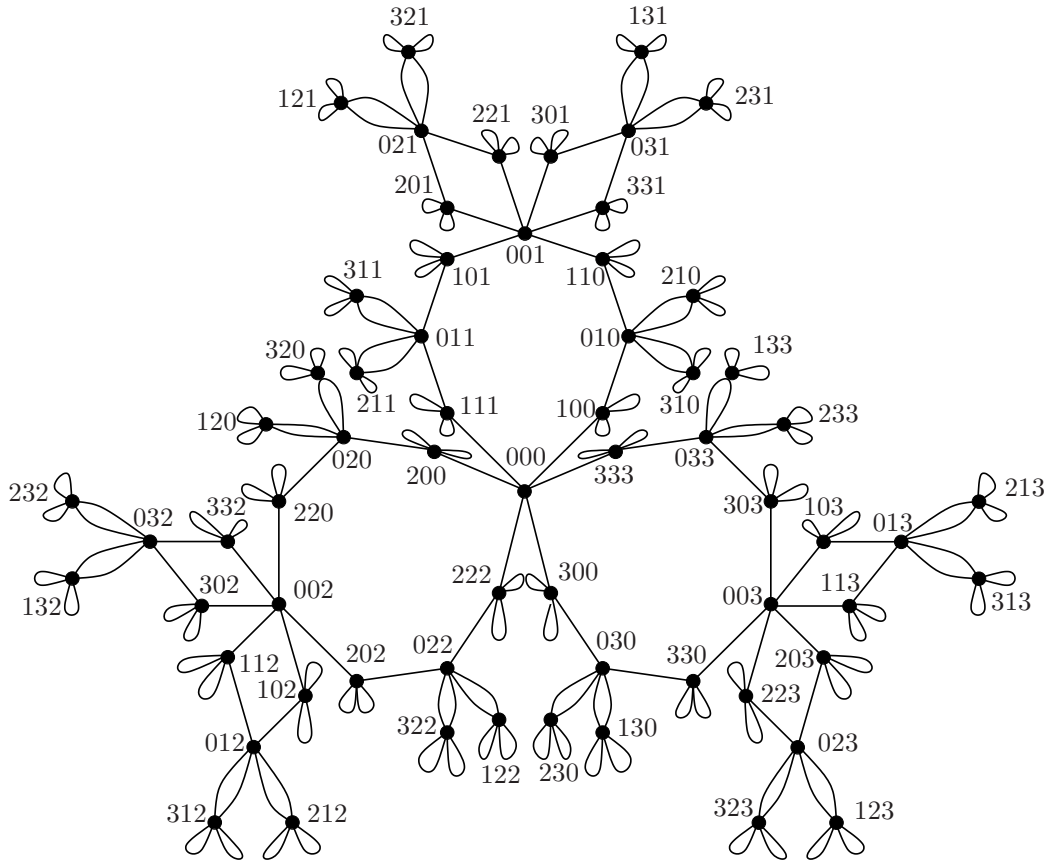}
\end{center}\caption{The Schreier graph $\Gamma_3$ associated with $\mathcal{G}_{S_3}$.} \label{alf2}
\end{figure}

\begin{theorem}\label{propspettro}
For each $n \geq 1$, the following factorization of the characteristic polynomial $P_n(\lambda)$ holds:
\begin{eqnarray}\label{riformula}
P_n(\lambda) = (\lambda -6) \cdot \prod_{i=0}^{n-1}\left(f^{\circ i}(\lambda) +2\right) \cdot \prod_{i=0}^{n-1}\left(f^{\circ i}(\lambda) -4\right)^{2\cdot 4^{n-i-1}},
\end{eqnarray}
where $f^{\circ i}(\lambda) =  \underbrace{f(f(\ldots f(\lambda)))}_{i \textrm{ times}}$. In particular, the adjacency spectrum of the graph $\Gamma_n$ is
$$
\{6\}\sqcup  \left(\bigcup_{i=0}^{n-1} f^{-i}(-2)\right) \sqcup \left(\bigcup_{i=0}^{n-1} (f^{-i}(4))^{2\cdot 4^{n-i-1}}\right).
$$
\end{theorem}
\begin{proof}
Let $f(\lambda)= \lambda^2-4\lambda-6$ as in Theorem \ref{th13}. The factorization given in Eq. \eqref{riformula} can be proved by induction on $n$, using the recurrence
$$
\left\{
  \begin{array}{ll}
    P_{n+1}(\lambda) & =  (\lambda-4)^{2\cdot 4^n} P_n(f(\lambda)) \\
    P_1(\lambda) & =  (\lambda-6)(\lambda+2)(\lambda-4)^2
  \end{array}
\right.
$$
obtained in Theorem \ref{th13}, and using the fact that
$$
f(\lambda) - 6 =  \lambda^2-4\lambda-12 = (\lambda -6)(\lambda +2),
$$
which also implies that the iterated backward orbit of $6$ under $f$ can be written as
\begin{eqnarray*}
f^{-n}(6) = \{6\} \sqcup \left(\bigcup_{i=0}^{n-1} f^{-i}(-2)\right), \qquad \forall \ n \geq 1.
\end{eqnarray*}
The claim about the adjacency spectrum follows.
\end{proof}

\begin{remark}\label{16nov}
The eigenvalues of $A_n$ given in Theorem \ref{propspettro} can be described more explicitly. In particular, a direct computation gives
$$
f^{-1}(-2) = \left\{2\pm 2\sqrt{2}\right\}, \qquad f^{-2}(-2) = \left\{2\pm \sqrt{12\pm 2\sqrt{2}}\right\}
$$
and, in general, it can be shown by induction that
$$
f^{-n}(-2)= \left\{2\pm \sqrt{12 \pm \sqrt{12 \pm \sqrt{\ldots \pm 2\sqrt{2}}}}\right\},   \ n\geq 1
$$
where the double sign $\pm$ occurs $n$ times. Similarly, one has
$$
f^{-1}(4) = \left\{2\pm \sqrt{14}\right\}, \qquad f^{-2}(4) = \left\{2\pm \sqrt{12\pm \sqrt{14}}\right\}
$$
and in general
$$
f^{-n}(4)= \left\{2\pm \sqrt{12 \pm \sqrt{12 \pm\sqrt{ \ldots \pm \sqrt{14}}}}\right\},  \ n\geq 1
$$
where also in this case the double sign $\pm$ occurs $n$ times. In Fig. \ref{figspettro} the histogram of the spectrum of the Schreier graph $\Gamma_6$ of the group $\mathcal{G}_{S_3}$ (in logarithmic scale) is represented.
\end{remark}

\begin{lemma}\label{nested}
Let $a,b>0$ and let $\{a_n\}_{n\geq 1}$ be the sequence defined by recursion as
$$
\left\{
  \begin{array}{ll}
a_1=\sqrt{b}\\
a_{n+1} = \sqrt{a+a_n}, \qquad n\geq 1.
  \end{array}
\right.
$$
If $\sqrt{b} < \frac{1}{2}(1+\sqrt{1+4a})$, then the sequence $\{a_n\}_{n\geq 1}$ is increasing and
$$
\lim_{n\to \infty}a_n = \frac{1}{2}(1+\sqrt{1+4a}).
$$
\end{lemma}
\begin{proof}
It is easy to show by induction that the sequence $\{a_n\}_{n\geq 1}$ is increasing and bounded, so that it admits a finite limit $\ell$. By squaring, one can see that such a limit $\ell$ must satisfy the equation $\ell^2-\ell -a=0$, whose solutions are $\ell = \frac{1\pm \sqrt{1+4a}}{2}$. The solution corresponding to the sign $-$ cannot be accepted, since it must be $\ell >0$, and we get the claim.
\end{proof}

\begin{theorem}\label{th23}
The spectrum of each infinite Schreier graph $\Gamma$ of $\mathcal{G}_{S_3}$ is the closure of the set of
points
$$
\{4\} \cup \left\{ 2\pm \underbrace{\sqrt{12 \pm \sqrt{12 \pm\sqrt{ \ldots \pm \sqrt{14}}}}}_{n \textrm{ times }},\ n \geq 1\right\}.
$$
This set is the union of a Cantor set of zero Lebesgue measure which is symmetric about $2$ and a
countable collection of isolated points supporting the KNS spectral measure $\mu$, which is discrete and which has value $\frac{1}{2\cdot 4^n}$ at the points whose definition involves $n$ radicals, for $n\geq 1$, and value $\frac{1}{2}$ at the point $4$.
\end{theorem}
\begin{proof}
Since the group $\mathcal{G}_{S_3}$ is amenable, the spectrum of each infinite Schreier  graph $\Gamma$ of $\mathcal{G}_{S_3}$ is given by
$$
\overline{\{6\}\sqcup  \left(\bigcup_{i=0}^{\infty} f^{-i}(-2)\right) \sqcup \left(\bigcup_{i=0}^{\infty} f^{-i}(4)\right)}.
$$
Let us investigate the dynamics of the quadratic map $f(\lambda)= \lambda^2-4\lambda -6$. As $f'(\lambda) = 2\lambda-4$, the unique
critical point of $f$ is $\lambda_0 = 2$. Therefore, the critical value $f(2) = -10$ is the unique value of $x$ such that the equation $f(\lambda) = x$ has a double root.

Now observe that Lemma \ref{nested} returns the limit value $4$ for $a=12$. It follows that, for each $n$, the spectrum of $\Gamma_n$ is contained in the interval $[-2,6]$. Now, it is easy to check that
$$
f^{-1}[-2, 6] = [-2, 2-2\sqrt{2}]\cup [2+2\sqrt{2},6] \subseteq [-2, 6].
$$
Since the critical value $-10\not\in f^{-1}[-2, 6] \subseteq [-2, 6]$ it follows that, for any
value of $x$ in $[-2, 6]$, the entire backward orbit $f^{-i}(x)$ is still contained in $[-2, 6]$ and
the sets $f^{-i}(x)$, for each $i \geq 0$, consist of $2^i$ distinct real numbers. Moreover it is known that, for such $x$, the sets $f^{-i}(x)$ are mutually disjoint for $i\geq 0$, provided $x$ is not a periodic point (a point $x$ is periodic if $f^k(x) = x$ for some positive integer $k$).

In our case, the forward orbit of $4$ under $f$ goes to $\infty$, so that $4$ is not a periodic
point and the sets $f^{-i}(4)$ are mutually disjoint, for $i\geq 0$. On the other hand, since $f(6) = 6$, so that $6$ is a fixed point for $f$,
the point $-2$ is not periodic and the sets $f^{-i}(-2)$ are mutually disjoint for $i\geq 0$. In particular, it follows that the number of distinct eigenvalues of the graph $\Gamma_n$ is
$$
1 + 2\sum_{i=0}^{n-1}2^i = 2^{n+1}-1, \qquad \textrm{for each } n\geq 1.
$$
Recall now that a periodic point $x$ of $f$ is repelling if $|f'(x)| > 1$. Since $f(6) = 6$ and
$f'(6) = 8>0$, the point $6$ is a repelling fixed point for the polynomial $f$. This implies
that the backward orbit $\{6\} \sqcup \left(\bigcup_{i=0}^{\infty} f^{-i}(-2)\right)$
of $6$ is in the Julia set $J$ of $f$, which is, by definition, the closure of the set of repelling periodic points of $f$ \cite{devaney}.\\
\indent On the other hand, the value $4$ is not in the Julia set, since its
forward orbit goes to $\infty$, and therefore the set $\bigcup_{i=0}^{\infty} (f^{-i}(4))$ is a countable
set of isolated points that accumulates to the Julia set $J$. It follows that the spectrum of $\Gamma$ is given by
$$
\overline{\bigcup_{i=0}^{\infty} f^{-i}(4)},
$$
where, for each $i$, the set $f^{-i}(4)$ has been described in Remark \ref{16nov}. Notice that the Julia set $J$ of $f$ is a Cantor set, since the map $f$ is conjugate via the map $F(z) = z+2$ to the quadratic map
$$
z \mapsto z^2-12,
$$
that is, $(F^{-1}\circ f \circ F) = z^2-12$, and $-12<-2$ (see Section 3.2 in \cite{devaney}).
Recall that the KNS spectral measure $\mu$ is limit of the
counting measures $\mu_n$ defined for $\Gamma_n$ as in Eq. \eqref{mcountingm}. We also know that in the spectrum of $\Gamma_n$, each eigenvalue in $\{6\} \sqcup \left(\bigcup_{i=0}^{n-1} f^{-i}(-2)\right)$ has multiplicity $1$, whereas each
eigenvalue  in $f^{-i}(4)$ has multiplicity $2\cdot 4^{n-i-1}$ for each $i$. Now
$$
\lim_{n\to \infty} \frac{2\cdot 4^{n-i-1}}{4^n} = \frac{1}{2\cdot 4^i}, \textrm{ for each } i \geq 0.
$$
Being $\sum_{i=0}^\infty \frac{2^i}{2\cdot 4^i}=1$, the KNS spectral measure is discrete and concentrated at these eigenvalues.
\end{proof}

\begin{figure}[h]
\begin{center}
\includegraphics[width=1\textwidth]{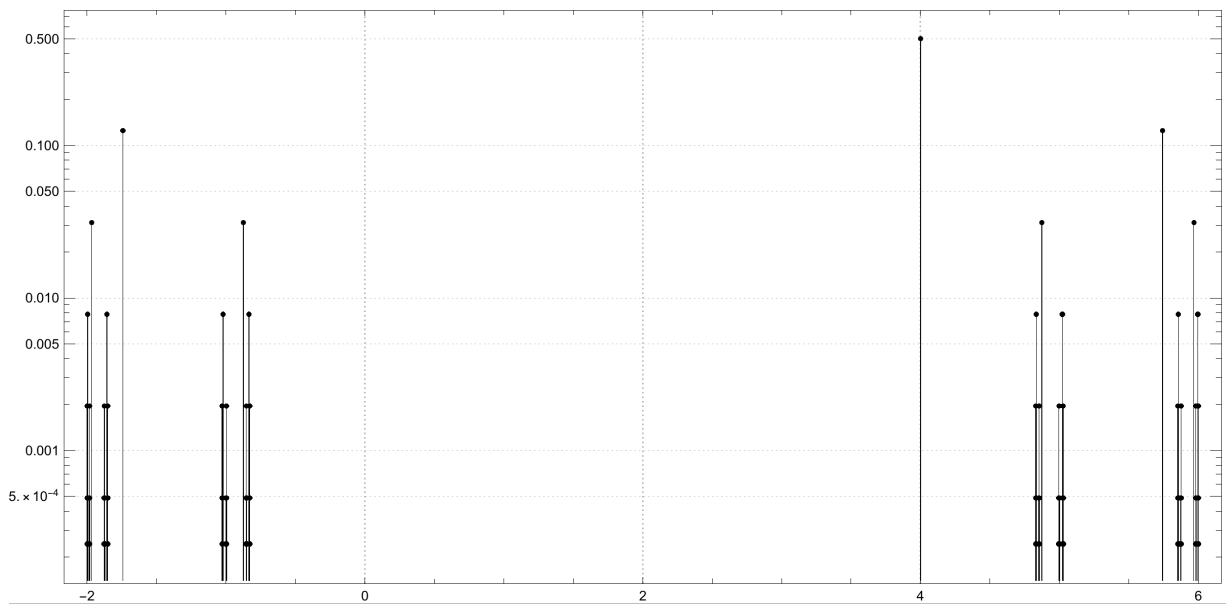}
\end{center}   \caption{The histogram of the spectrum of the Schreier $\Gamma_6$ of the group $\mathcal{G}_{S_3}$.}\label{figspettro}
\end{figure}

The Ihara zeta function $\zeta_n(t)$ of the Schreier graph $\Gamma_n$ of $\mathcal{G}_{S_3}$ satisfies the equation
$$
\zeta_n(t) =(1-t^2)^{-2\cdot 4^n}\det(1-tA_n+5t^2)^{-1},
$$
where $A_n$ is the adjacency matrix of $\Gamma_n$. When passing to the limit, the following integral presentation holds:
$$
\ln \zeta_\Gamma(t) = -2\ln(1-t^2)-\int_{-1}^1 \ln(1-6t\lambda +5t^2) d\mu(\lambda), \quad \forall \ t : |t|<\frac{1}{5}
$$
where $\mu$ is the KNS spectral measure and $\lambda$ runs over the normalized spectrum of $\Gamma$. In our case we get, for each $t$ such  that $|t|<\frac{1}{5}$:
\begin{eqnarray*}
\ln \zeta_\Gamma(t) &=& -2\ln(1-t^2)- \frac{1}{2}\ln(1-4t+5t^2)\\
 &-&\frac{1}{2}\sum_{i=1}^\infty \frac{1}{4^{i}}\ln \left(1-t\left(\underbrace{2\pm\sqrt{12\pm \sqrt{12\pm\sqrt{ \ldots \pm\sqrt{14}}}}}_{\pm \ i \textrm{ times }}\right)+5t^2\right).   \nonumber
\end{eqnarray*}

\subsection{The general case}\label{secspectralp}
Let $p\geq 1$ be an integer number. The aim of this subsection is to generalize what we have seen in the previous subsection for the graph $S_3$ to the more general context of a star graph $S_p$ on $p+1$ vertices. We will not give all the details presented in the case $p=3$.

The star automaton group $\mathcal{G}_{S_p}$ is the group generated by $p$ automorphisms $e_i$, $i=1,\ldots, p$, having the following self-similar representation:
\begin{eqnarray}\label{gigi}
e_i=(e_i,id,\ldots, id)(0i) \quad \textrm{ for each } i=1,\ldots, p.
\end{eqnarray}
Notice that
$$
e_i^{-1}=(id,\ldots, id,\underbrace{e_i^{-1}}_{(i+1)\textrm{-th place}}, id,\ldots, id)(0i) \quad \textrm{ for each } i=1,\ldots, p.
$$
The group $\mathcal{G}_{S_p}$ acts on the rooted tree $T_{p+1}$. The $n$-th level of such tree consists of $(p+1)^n$ vertices, identified with the set of words of length $n$ over the alphabet $\{0,1,\ldots, p\}$. As a consequence, the $n$-th Schreier graph is a regular graph of degree $2p$ on $(p+1)^n$ vertices, and its adjacency matrix $A_n$ is a symmetric matrix of size $(p+1)^n$. We will adopt the notation $\Gamma^p_n$ to denote the $n$-th Schreier graph associated with the action of $\mathcal{G}_{S_p}$.
The following theorem holds.

\begin{theorem}\label{th1p}
Let $P_n(\lambda)$ be the characteristic polynomial of the adjacency matrix $A_n$ of the Schreier graph $\Gamma^p_n$ of the group $\mathcal{G}_{S_p}$, for each $n\geq 1$. Then
\begin{eqnarray*}
P_{n+1}(\lambda) = (\lambda-2(p-1))^{(p-1)\cdot (p+1)^n} P_n(f_p(\lambda)),
\end{eqnarray*}
with $f_p(\lambda) = \lambda^2-2(p-1)\lambda -2p$ and $P_1(\lambda) = (\lambda-2p)(\lambda+2)(\lambda-2(p-1))^{p-1}$.
\end{theorem}

Moreover, one can still define $\Gamma^p_0$ to be the graph consisting of a single vertex endowed with $p$ loops. In this way, one has $P_0(\lambda)= \lambda -2p$ and the equation
$$
P_1(\lambda) = (\lambda-2(p-1))^{p-1} \cdot P_0(f_p(\lambda)),
$$
is still satisfied.

\begin{theorem}\label{propspettrop}
For each $n \geq 1$, the following factorization of the characteristic polynomial $P_n(\lambda)$ holds:
$$
P_n(\lambda) = (\lambda -2p) \cdot \prod_{i=0}^{n-1}\left(f_p^{\circ i}(\lambda) +2\right) \cdot \prod_{i=0}^{n-1}\left(f_p^{\circ i}(\lambda) -2(p-1)\right)^{(p-1)\cdot (p+1)^{n-i-1}},
$$
where $f_p^{\circ i}(\lambda) =  \underbrace{f_p(f_p(\ldots f_p(\lambda)))}_{i \textrm{ times}}$. In particular, the adjacency spectrum of the graph $\Gamma^p_n$ is
$$
\{2p\}\sqcup  \left(\bigcup_{i=0}^{n-1} f_p^{-i}(-2)\right) \sqcup \left(\bigcup_{i=0}^{n-1} (f_p^{-i}(2(p-1)))^{(p-1)\cdot (p+1)^{n-i-1}}\right).
$$
\end{theorem}
\begin{proof}
The proof proceeds as in Theorem \ref{propspettro} and uses the fact that
\begin{eqnarray*}
f_p^{-n}(2p) = \{2p\} \sqcup \left(\bigcup_{i=0}^{n-1} f_p^{-i}(-2)\right),
\end{eqnarray*}
because it holds $\lambda^2-2(p-1)\lambda -4p = (\lambda-2p)(\lambda + 2)$.
\end{proof}

In Fig. \ref{alf3} and Fig. \ref{alf4} the Schreier graphs $\Gamma^2_n$, for $n=1,2,3,4$, associated with the group $\mathcal{G}_{S_2}$ are depicted. Vertices are labeled by words in $\{0,1,2\}^n$.
\begin{figure}[h]
\begin{center}           \footnotesize
\psfrag{1}{$1$}\psfrag{2}{$2$}\psfrag{0}{$0$}
\psfrag{00}{$00$}\psfrag{01}{$01$}\psfrag{02}{$02$}
\psfrag{10}{$10$}\psfrag{11}{$11$}\psfrag{12}{$12$}
\psfrag{20}{$20$}\psfrag{21}{$21$}\psfrag{22}{$22$}

\psfrag{000}{$000$}\psfrag{001}{$001$}\psfrag{002}{$002$}
\psfrag{010}{$010$}\psfrag{011}{$011$}\psfrag{012}{$012$}
\psfrag{020}{$020$}\psfrag{021}{$021$}\psfrag{022}{$022$}

\psfrag{100}{$100$}\psfrag{101}{$101$}\psfrag{102}{$102$}
\psfrag{110}{$110$}\psfrag{111}{$111$}\psfrag{112}{$112$}
\psfrag{120}{$120$}\psfrag{121}{$121$}\psfrag{122}{$122$}

\psfrag{200}{$200$}\psfrag{201}{$201$}\psfrag{202}{$202$}
\psfrag{210}{$210$}\psfrag{211}{$211$}\psfrag{212}{$212$}
\psfrag{220}{$220$}\psfrag{221}{$221$}\psfrag{222}{$222$}

\includegraphics[width=0.9\textwidth]{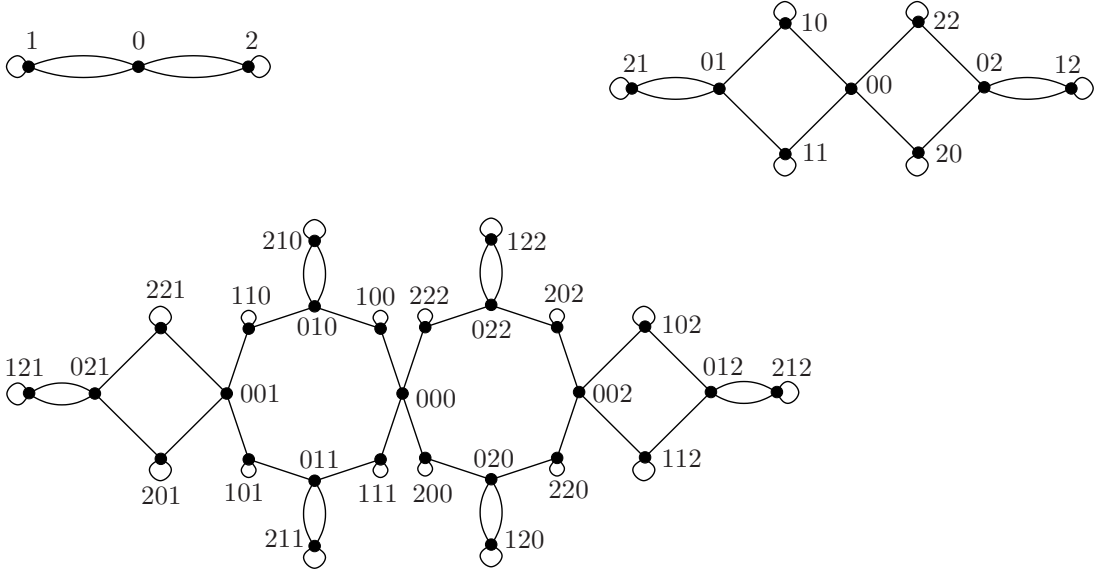}
\end{center}\caption{The Schreier graphs $\Gamma^2_n$ associated with $\mathcal{G}_{S_2}$, for $n=1,2,3$.} \label{alf3}
\end{figure}

\begin{figure}[h]
\begin{center}
\tiny
\psfrag{0000}{$0000$}\psfrag{0010}{$0010$}\psfrag{0020}{$0020$}
\psfrag{0100}{$0100$}\psfrag{0110}{$0110$}\psfrag{0120}{$0120$}
\psfrag{0200}{$0200$}\psfrag{0210}{$0210$}\psfrag{0220}{$0220$}

\psfrag{0001}{$0001$}\psfrag{0011}{$0011$}\psfrag{0021}{$0021$}
\psfrag{0101}{$0101$}\psfrag{0111}{$0111$}\psfrag{0121}{$0121$}
\psfrag{0201}{$0201$}\psfrag{0211}{$0211$}\psfrag{0221}{$0221$}

\psfrag{0002}{$0002$}\psfrag{0012}{$0012$}\psfrag{0022}{$0022$}
\psfrag{0102}{$0102$}\psfrag{0112}{$0112$}\psfrag{0122}{$0122$}
\psfrag{0202}{$0202$}\psfrag{0212}{$0212$}\psfrag{0222}{$0222$}

\psfrag{1000}{$1000$}\psfrag{1010}{$1010$}\psfrag{1020}{$1020$}
\psfrag{1100}{$1100$}\psfrag{1110}{$1110$}\psfrag{1120}{$1120$}
\psfrag{1200}{$1200$}\psfrag{1210}{$1210$}\psfrag{1220}{$1220$}

\psfrag{1001}{$1001$}\psfrag{1011}{$1011$}\psfrag{1021}{$1021$}
\psfrag{1101}{$1101$}\psfrag{1111}{$1111$}\psfrag{1121}{$1121$}
\psfrag{1201}{$1201$}\psfrag{1211}{$1211$}\psfrag{1221}{$1221$}

\psfrag{1002}{$1002$}\psfrag{1012}{$1012$}\psfrag{1022}{$1022$}
\psfrag{1102}{$1102$}\psfrag{1112}{$1112$}\psfrag{1122}{$1122$}
\psfrag{1202}{$1202$}\psfrag{1212}{$1212$}\psfrag{1222}{$1222$}

\psfrag{2000}{$2000$}\psfrag{2010}{$2010$}\psfrag{2020}{$2020$}
\psfrag{2100}{$2100$}\psfrag{2110}{$2110$}\psfrag{2120}{$2120$}
\psfrag{2200}{$2200$}\psfrag{2210}{$2210$}\psfrag{2220}{$2220$}

\psfrag{2001}{$2001$}\psfrag{2011}{$2011$}\psfrag{2021}{$2021$}
\psfrag{2101}{$2101$}\psfrag{2111}{$2111$}\psfrag{2121}{$2121$}
\psfrag{2201}{$2201$}\psfrag{2211}{$2211$}\psfrag{2221}{$2221$}

\psfrag{2002}{$2002$}\psfrag{2012}{$2012$}\psfrag{2022}{$2022$}
\psfrag{2102}{$2102$}\psfrag{2112}{$2112$}\psfrag{2122}{$2122$}
\psfrag{2202}{$2202$}\psfrag{2212}{$2212$}\psfrag{2222}{$2222$}
\includegraphics[width=1\textwidth]{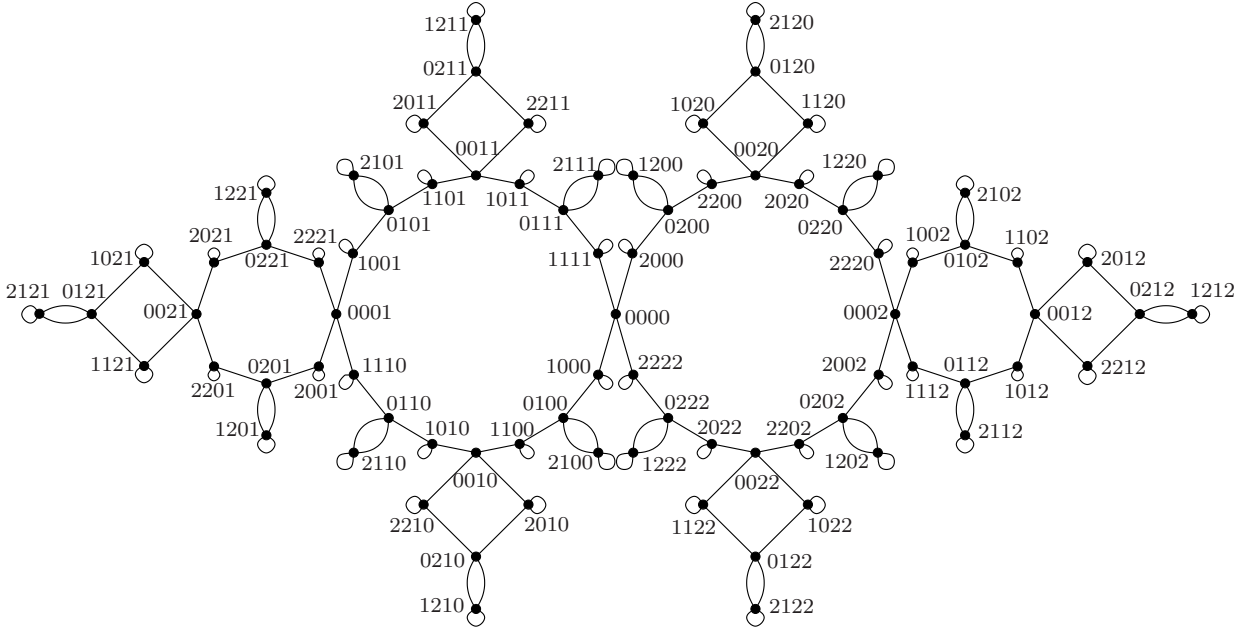}
\end{center}\caption{The Schreier graph $\Gamma^2_4$ associated with $\mathcal{G}_{S_2}$.} \label{alf4}
\end{figure}

\begin{theorem}\label{thmspettrop}
Let $p\geq 2$. The spectrum of each infinite Schreier graph $\Gamma^p$ of $\mathcal{G}_{S_p}$ is the closure of the set of
points
$$
\{2(p-1)\} \cup \left\{p-1\pm \underbrace{\sqrt{p^2+p \pm \sqrt{p^2+p \pm\sqrt{ \ldots \pm \sqrt{p^2+2p-1}}}}}_{n \textrm{ times }},\ n \geq 1\right\}.
$$
This set is the union of a Cantor set of zero Lebesgue measure which is symmetric about $p-1$ and a
countable collection of isolated points supporting the KNS spectral measure $\mu$, which is discrete and which has value $\frac{p-1}{(p+1)^{n+1}}$ at the points whose definition involves $n$ radicals, for $n\geq 1$, and value $\frac{p-1}{p+1}$ at the point $2(p-1)$.
 \end{theorem}
\begin{proof}
The proof proceeds as in the case $p=3$. The spectrum of  $\Gamma^p$ is given by
$$
\overline{\{2p\}\sqcup  \left(\bigcup_{i=0}^{\infty} f_p^{-i}(-2)\right) \sqcup \left(\bigcup_{i=0}^{\infty} f_p^{-i}(2(p-1))\right)},
$$
with $f_p(\lambda)= \lambda^2-2(p-1)\lambda -2p$. A direct computation gives
$$
f_p^{-n}(-2)= \left\{p-1 \pm \sqrt{p^2+p \pm \sqrt{p^2+p \pm\sqrt{ \ldots \pm \sqrt{p^2-1}}}}\right\},   \ n\geq 1
$$
$$
f_p^{-n}(2(p-1))= \left\{p-1\pm \sqrt{p^2+p \pm \sqrt{p^2+p \pm\sqrt{ \ldots \pm \sqrt{p^2+2p-1}}}}\right\},  \ n\geq 1
$$
where the double sign $\pm$ occurs $n$ times. By using Lemma \ref{nested}, it is easy to check that, for each $n$, the spectrum of $\Gamma^p_n$ is contained in the interval $[-2,2p]$. Here, the countable
set of isolated points that accumulates to the Julia set $J$ of $f_p$ is the set $\bigcup_{i=0}^{\infty} (f_p^{-i}(2(p-1)))$. Notice that the Julia set $J$ of $f_p$ has the structure of a Cantor set, since the map $f_p$ is conjugate via the map $F_p(z) = z+(p-1)$ to the quadratic map
$$
z \mapsto z^2-p(p+1),
$$
and $-p(p+1)<-2$ for every $p\geq 2$.
\end{proof}

The Ihara zeta function $\zeta_n(t)$ of the Schreier graph $\Gamma^p_n$ satisfies the equation
$$
\zeta_n(t) =(1-t^2)^{-(p-1)\cdot (p+1)^n}\det(1-tA_n+(2p-1)t^2)^{-1},
$$
where $A_n$ is the adjacency matrix of $\Gamma^p_n$. When passing to the limit, the following integral presentation holds:
$$
\ln \zeta_{\Gamma^p}(t) = -(p-1)\ln(1-t^2)-\int_{-1}^1 \ln(1-2p t\lambda +(2p-1)t^2) d\mu(\lambda), \quad \forall \ t : |t|<\frac{1}{2p-1}
$$
where $\mu$ is the KNS spectral measure and $\lambda$ runs over the normalized spectrum of $\Gamma^p$. In particular we obtain, for each $t$ such that $|t|<\frac{1}{2p-1}$:
\begin{eqnarray*}
\ln \zeta_{\Gamma^p}(t) &=& -(p-1)\ln(1-t^2)- \frac{p-1}{p+1}\ln(1-2(p-1)t+(2p-1)t^2) - \frac{p-1}{p+1}\cdot\\
 &\cdot&\sum_{i=1}^\infty \frac{1}{(p+1)^{i}}\ln \left(1-t\left(p-1\pm\sqrt{p^2+p\pm \sqrt{ \ldots \pm\sqrt{p^2+2p-1}}}\right)+(2p-1)t^2\right).   \nonumber
\end{eqnarray*}

\begin{remark}\rm
In Theorem \ref{thmspettrop} we have supposed $p\geq 2$. In fact, for $p=1$, the star $S_1$ is the path graph $P_2$ on $2$ vertices. The associated star automaton group $\mathcal{G}_{S_2}$ is the group acting on the binary rooted tree $T_2$ generated by the automorphism $a$ having the self-similar representation
$$
a=(a, id)(01),
$$
which is isomorphic to the group $\mathbb{Z}$ and which is classically known as \emph{Adding machine}. For each $n\geq 1$, the $n$-th Schreier graph $\Gamma^1_n$ is a cycle on $2^n$ vertices. Moreover, one has $f_1 =\lambda^2-2$, and the Julia set of this quadratic map is the whole interval $[-2,2]$.\\
\indent The case $p=2$ corresponds to the path graph $P_3$ on $3$ vertices. The associated star automaton group $\mathcal{G}_{S_3}$ is known as \emph{Tangled odometer} and it is the group acting on the rooted ternary tree $T_3$ generated by the automorphisms $a$ and $b$ having the following self-similar representation:
$$
a=(a, id,id)(01) \qquad b=(b,id,id)(02).
$$
This case has been investigated in \cite{articolo0}, where the family of groups associated with the path graph $P_k$, for every $k \geq 2$, has been treated in  detail. Notice that, for $p=2$, our spectral results recover the ones given for this group in \cite[Theorem 6.3]{fromtangled}.
\end{remark}

\section{Schreier graphs of star automaton groups} \label{sectionschreier}

In this section we give a complete classification of the infinite Schreier graphs associated with the star automaton group $\mathcal{G}_{S_p}$. By complete classification we mean the following: we have already remarked that, given an infinite sequence $w=x_1x_2x_3\ldots \in X^{\infty}$, one can define the rooted graph $(\Gamma_w,w)$ as limit of the sequence of rooted graphs $(\Gamma_n,x_1\ldots x_n)$. Now we can forget the root and consider the corresponding (unrooted) infinite Schreier graph. We want to the describe the isomorphism classes of such infinite graphs arising from the action of the star automaton group on $X^{\infty}$. In what follows, given a subgraph $\Theta$ of $\Gamma^p_n$ we denote by $\Theta w$ the set of vertices obtained by appending the word $w\in X^{\ast}\cup X^{\infty}$ to the vertices of $\Theta$. When it is clear from the context, with abuse of notation, we identify a set of vertices of a graph with its induced subgraph. The geodesic distance (or distance for short) between the vertices $u,v$ is denoted by $d(u,v)$.

\begin{definition}\label{defcomp}
\begin{enumerate}
\item Two infinite sequences $\xi=x_1x_2x_3\ldots$ and $\eta=y_1y_2y_3\ldots$ in $X^{\infty}$ are cofinal if there exists $k\in\mathbb{N}$ such that $x_n=y_n$ for any $n\geq k$.
\item Two sequences $\{x_i\}_{i\in \mathbb{N}}$ and $\{y_i\}_{i\in \mathbb{N}}$ of integers are compatible if there exist $l,h\in\mathbb{N}$ such that $x_{l+n}=y_{h+n}$ for any $n\in\mathbb{N}$.
\end{enumerate}
\end{definition}

In other words, two infinite words over $X$ are cofinal if they differ only for a finite prefix. In this case we write $\xi\sim\eta$. The cofinality is an equivalence relation, and we denote by $Cof(\xi)$ the equivalence class of words cofinal to $\xi$. 
Two sequences are compatible if they coincide after removing from them some terms (possibly a different number of them). Notice that also being compatible is an equivalence relation.

\subsection{Finite Schreier graphs} \label{sectionfinite}
From now on we fix a star $S_p$ and we use the same representation of  Fig. \ref{figureS3}. In this case $X=\{0,1,\ldots, p\}$, where 0 is the vertex of degree $p$. We denote by $e_i$ the (directed) edge connecting 0 to $i\in \{1,\ldots, p\}$, so that $e_i=(e_i, id,\ldots, id)(0 i)$ (see Eq. \eqref{gigi}). We will denote by $\Gamma^p_n$ the $n$-th Schreier graph of the group $\mathcal{G}_{S_p}$. \\
Observe that the generator $e_i$ of $\mathcal{G}_{S_p}$ acts like an adding machine on the set $\{0,i\}^{\ast}$. More precisely, when we let it act on a vertex of type $0^tjw$, with $j\neq 0,i$ and $|w|=n-t-1$, we obtain a cycle of length $2^t$ whose vertex set is the whole set $\{0,i\}^tjw$. Let us denote by $C^i_n$ the (maximal) cycle of length $2^n$ labeled by $e_i$ for $i=1,\ldots, p$. Notice that the maximal cycles in $\Gamma^p_n$ are exactly those generated by the $e_i$'s and containing $0^n$.

\begin{example}\rm
In Fig. \ref{alfn}, which represents the Schreier graph $\Gamma_3^3$, the three maximal cycles have length $8$. With respect to Eq. \eqref{generators3}, one has: $e_1=a$, $e_2=b$, $e_3=c$. In particular:
\begin{itemize}
\item the cycle $C^1_3$, containing the adjacent vertices $000$ and $111$, is obtained by letting $a$ act on the vertex $000$;
\item the cycle $C^2_3$, containing the adjacent vertices $000$ and $222$, is obtained by letting $b$ act on the vertex $000$;
\item the cycle $C^3_3$, containing the adjacent vertices $000$ and $333$, is obtained by letting $c$ act on the vertex $000$.
\end{itemize}
\end{example}

\begin{lemma}\label{lemma1}
If $u,v$ are adjacent vertices in $\Gamma^p_n$ then the vertices $uw$ and $vw$ are adjacent in $\Gamma^p_{n+|w|}$ for any $w\in X^{\ast}$ with the only exception, for $i=1,\ldots, p$, given by $\{u,v\}= \{0^n, i^n\}$ and $w$ starting with $0$ or $i$.
\end{lemma}
\begin{proof}
It is enough to notice that, if $u,v$ are adjacent vertices in $\Gamma^p_n$, then there exists $i$ such that $e_i(u)=v$, i.e., a directed path in the generating automaton labeled by $u$ and $v$ and starting from the state $e_i$. Such a path must either end up in the sink (when $\{u,v\}\neq \{0^n, i^n\}$) or end up in $e_i$ (when $\{u,v\}=\{0^n, i^n\}$). In the first case, we can append to $u$ any $w\in X^{\ast}$ in such a way that $e_i(uw)=vw$. In the second case, if $w$ starts with a letter $j\neq 0,i$, the path labeled by $0^nj$ and $e_i(0^nj)=i^nj$ ends up in the trivial state. Hence also in the case $\{u,v\}=\{0^n, i^n\}$ and $w$ not starting by $0$, $i$ one has that $uw$ and $vw$ are adjacent in $\Gamma^p_{n+|w|}$.
\end{proof}

\begin{remark}
Lemma \ref{lemma1} implies that any cycle $C$ in $\Gamma^p_n$ labeled by $e_i$ and different from $C_n^i$ appears $(p+1)^k$ times in the Schreier graph $\Gamma^p_{n+k}$, with vertices $Cw$ for any $w\in X^k$. 
 The same can be said for cycles of the form $C_n^iv$ where $v$ does not start with $0$ or $i$.
\end{remark}

From Lemma \ref{lemma1} we deduce that, passing from $\Gamma^p_n$ to $\Gamma^p_{n+1}$, each cycle in $\Gamma^p_n$ is preserved just by adding to all its vertices the same letter $k\in X$ except for some of the $p$ maximal cycles $C^i_n$. In fact $C^i_nj$, with $j\neq 0, i$ also corresponds to a subgraph in $\Gamma^p_{n+1}$ that is a copy of $C_n^i$, whereas $C^i_n 0$ and $C^i_n i$ correspond to the two halves of the new maximal cycle $C^i_{n+1}$ of $\Gamma^p_{n+1}$.
\begin{example}\rm
Look  at Fig. \ref{alf1} and Fig. \ref{alf2}, where $p=3$. We have that the maximal cycle $C^1_2$ in $\Gamma^3_2$ produces the cycles $C^1_22$ and $C^1_23$ of length $4$ in $\Gamma^3_3$, which are attached to the vertex $002$ and $003$, respectively. On the other hand, the cycles $C^1_20$ and $C^1_21$ do not appear in $\Gamma^3_3$, but they constitute the two halves of the maximal cycle $C^1_3$ (the edge connecting $001$ and $111$ and the edge connecting $000$ and $110$ do not appear, whereas two new edges connecting the vertices $001$ and $110$, and the vertices $000$ and $111$, appear).
\end{example}

Recall that a \emph{cut-vertex} of a graph is a vertex whose deletion increases the number of connected components of the graph (see, for instance, \cite{bollobas}).
Following \cite[Proposition 4.7]{articolo0} we have that $0u$ is a cut-vertex in $\Gamma^p_n$ for any $u\in X^{n-1}$. In particular $0^n$ is a cut-vertex. Notice that, by removing $0^n$ from $\Gamma^p_n$, we obtain $p$ connected components that we call \emph{petals}. More precisely, the vertex $0^n$ belongs to the maximal cycle $C^i_n$, for each $i=1,\ldots, p$, and the connected component containing this maximal cycle generated by $e_i$ is called the $i$-th petal. One can show that the $i$-th petal consists of the set of vertices ending with a suffix $i0^k$, for $k=0,1,\ldots, n-1$. See, for instance, Fig. \ref{alfn}, representing the Schreier graph $\Gamma_3^3$, where the $1$-st petal is highlighted in the upper part of the graph.\\
\indent All other vertices of $\Gamma^p_n$, those beginning with $i\neq 0$, have $p-1$ loops corresponding to the actions of the generators $e_j$, with $j\neq i$ (we consider loops as cycles of length 1). In the remaining part of the paper, we will consider also such vertices as cut-vertices. In particular, it follows that $\Gamma^p_n$ has a cactus structure. In particular, the following lemma holds.

\begin{lemma}\label{coro}
The vertex $0^n$ is a cut-vertex belonging to $C_n^i$ for any $i=1,\ldots, p$.
Any vertex $v\in \{0,i\}^n\setminus\{0^n\}$ is a cut-vertex belonging to $C_n^i$ and to other $p-1$ cycles labeled by $e_j$, with $j\neq i$, whose size is $2^k$ if $v=0^kiv'$, with $0\leq k\leq n-1$.
\end{lemma}

\begin{definition}
Let $\Gamma^p_n$ be $n$-th Schreier graph of the group $\mathcal{G}_{S_p}$. Let $i\in \{1,\ldots, p\}$. The $n$-decoration $\mathcal{D}_n^i$ is the subgraph of $\Gamma^p_n$ obtained by removing from  $\Gamma^p_n$ the $i$-th petal.
\end{definition}

Notice that $\mathcal{D}_n^i$ contains $0^n$ and is connected. Basically, it is the union of the petals different from the $i$-th one together with the vertex $0^n$. Moreover $\mathcal{D}_n^i$ and $\mathcal{D}_n^j$ are isomorphic graphs for any $i,j$. When we are not interested in the specific decoration, but just in its structure, we only write $\mathcal{D}_n$.


From Lemma \ref{lemma1} and Lemma \ref{coro} it follows that $\mathcal{D}_n^iiw$ induces a subgraph in $\Gamma^p_{n+1+|w|}$ which is a copy of $\mathcal{D}_n^i$ via the map $viw\mapsto v$.
In particular $\mathcal{D}_n^ii$ is attached to the vertex $0^{n}i$ of the maximal cycle $C^i_{n+1}$ generated by $e_i$ in $\Gamma^p_{n+1}$. From this it follows that $\mathcal{D}_n^ii$ is a subgraph of $\mathcal{D}_{n+1}^j$ for every $j\neq i$. When we want to highlight the fact that its structure comes from the $n$-th level, we  say that such subgraph of $\Gamma^p_{n+1}$ is an $n$-decoration of $C^i_{n+1}$. By using an analogous argument, we deduce that the subgraphs $\mathcal{D}_n^iii$ and $\mathcal{D}_n^ii0$ inside $\Gamma^p_{n+2}$  are $n$-decorations attached to $C^i_{n+2}$ at the vertices $0^nii$ and $0^ni0$. By iterating this argument we can conclude that, for every $u\in\{0,i\}^m$, the subgraph $\mathcal{D}_n^iiu$ is an $n$-decoration in $\Gamma^p_{n+m+1}$ attached to $C^i_{n+m+1}$ at the vertex $0^niu$. Hence in $\Gamma^p_{n}$ we have attached to $C^i_{n}$:
\begin{itemize}
\item the decoration $\mathcal{D}_n^i$ at the vertex $0^n$;
\item one $(n-1)$-decoration given by $\mathcal{D}_{n-1}^ii$ at the vertex $0^{n-1}i$;
\item $2^{k}$ copies of an $(n-k-1)$-decoration given by $\mathcal{D}_{n-k-1}^iiu$ at the vertex $0^{n-k-1}iu$, with $k=1,\ldots, n-1$ for every $u\in \{0,i\}^k$.
\end{itemize}
Here, by $0$-decoration we mean a vertex with $p-1$ loops attached. In Fig. \ref{alfn}, representing the Schreier graph $\Gamma_3^3$, the $2$-decoration $\mathcal{D}_2^22$ and one $1$-decoration given by $\mathcal{D}^2_123$ are depicted, attached to the vertex $002$ and $023$, respectively.

\begin{remark}\label{remo}
Notice that the vertex $0^kiv'$ of $C_n^i$ has attached the $k$-decoration $\mathcal{D}_k^iiv'$.
\end{remark}

\begin{proposition}\label{prop1}
Let $\phi_n^i$ be the nontrivial automorphism of $C_n^i$ fixing $0^n$. Then for any $v\in C_n^i$, the vertices $v$ and $\phi_n^i(v)$ have attached decorations that are isomorphic. In particular $\phi_n^i(v)$ is the only vertex of $C_n^i$ satisfying $d(0^n,v)=d(0^n,\phi_n^i(v))$.
\end{proposition}
\begin{proof}
The vertices of $C_n^i$ can be identified with the numbers $0, \ldots, 2^n-1$ by using the binary expansion (from the left to the right) of such numbers by identifying $i$ with $1$. Notice that the automorphism $\phi_n^i$ is a reflection around the axis connecting $0^n$ and $0^{n-1}i$ and it acts in such a way that $v+\phi_n^i(v)\equiv 0$ mod $2^{n}$. In particular, if $u=0^kiv$, with $k\in\{0,1,\ldots, n-1\}$, then $\phi_n^i(u)=0^kiv'$, where $v'$ is the word obtained from $v$ by switching any $0$ to $i$ and viceversa. By Remark \ref{remo} such vertices have attached the same $(n-|v|-1)$-decoration. The claim follows.
\end{proof}
Any vertex of $\Gamma^p_n$ is a cut-vertex belonging to $p$ different cycles. If the vertex $u$ belongs to the $i$-th petal of $\Gamma^p_n$, then there is a unique path of cycles, connecting $u$ to $C_n^i$. The first cycle in this path is the one containing $u$ in the direction of $C_n^i$.  Notice that the path of cycles is not defined for $0^n$. From now on, we do not consider this vertex.

We denote by $\mathcal{P}_n^u=\{P_{1}^n(u), \ldots, P_{m_u}^n(u)\}$ the \emph{path of cycles} associated with $u\in \Gamma^p_n$. Notice that $ P_{m_u}^n(u)$ is $C_n^i$ if $u$ belongs to the $i$-th petal. Moreover we denote by $\mathcal{L}_n^u=\{L_{1}^n(u), \ldots, L_{m_u}^n(u)\}$ the set of the lengths of the cycles in $\mathcal{P}_n^u$, i.e., $ L_{i}^n(u)$ is the length of the cycle  $P_{i}^n(u)$. In what follows, with a small abuse of notation, we identify the graph $P_{i}^n(u)$ with its vertex set.

\begin{example}
Looking at the graph $\Gamma^3_3$ of Fig. \ref{alfn}, we have $\mathcal{L}_3^{121}=\{2,4,8\}$; $\mathcal{L}_3^{201}=\{4,8\}$; $\mathcal{L}_3^{130}=\{2,8\}$.
\begin{figure}[h]
\begin{center}    \footnotesize
\psfrag{000}{$000$}\psfrag{001}{$001$}\psfrag{002}{$002$}\psfrag{003}{$003$}
\psfrag{010}{$010$}\psfrag{011}{$011$}\psfrag{012}{$012$}\psfrag{013}{$013$}
\psfrag{020}{$020$}\psfrag{021}{$021$}\psfrag{022}{$022$}\psfrag{023}{$023$}
\psfrag{030}{$030$}\psfrag{031}{$031$}\psfrag{032}{$032$}\psfrag{033}{$033$}

\psfrag{100}{$100$}\psfrag{101}{$101$}\psfrag{102}{$102$}\psfrag{103}{$103$}
\psfrag{110}{$110$}\psfrag{111}{$111$}\psfrag{112}{$112$}\psfrag{113}{$113$}
\psfrag{120}{$120$}\psfrag{121}{$121$}\psfrag{122}{$122$}\psfrag{123}{$123$}
\psfrag{130}{$130$}\psfrag{131}{$131$}\psfrag{132}{$132$}\psfrag{133}{$133$}

\psfrag{200}{$200$}\psfrag{201}{$201$}\psfrag{202}{$202$}\psfrag{203}{$203$}
\psfrag{210}{$210$}\psfrag{211}{$211$}\psfrag{212}{$212$}\psfrag{213}{$213$}
\psfrag{220}{$220$}\psfrag{221}{$221$}\psfrag{222}{$222$}\psfrag{223}{$223$}
\psfrag{230}{$230$}\psfrag{231}{$231$}\psfrag{232}{$232$}\psfrag{233}{$233$}

\psfrag{300}{$300$}\psfrag{301}{$301$}\psfrag{302}{$302$}\psfrag{303}{$303$}
\psfrag{310}{$310$}\psfrag{311}{$311$}\psfrag{312}{$312$}\psfrag{313}{$313$}
\psfrag{320}{$320$}\psfrag{321}{$321$}\psfrag{322}{$322$}\psfrag{323}{$323$}
\psfrag{330}{$330$}\psfrag{331}{$331$}\psfrag{332}{$332$}\psfrag{333}{$333$}
\psfrag{P}{$1$-st petal}

\psfrag{A}{$C^1_3$}\psfrag{B}{$C^3_3$}\psfrag{C}{$C^2_3$}
\psfrag{D}{$\mathcal{D}_2^22$}\psfrag{d}{$\mathcal{D}^2_123$}
 \includegraphics[width=0.9\textwidth]{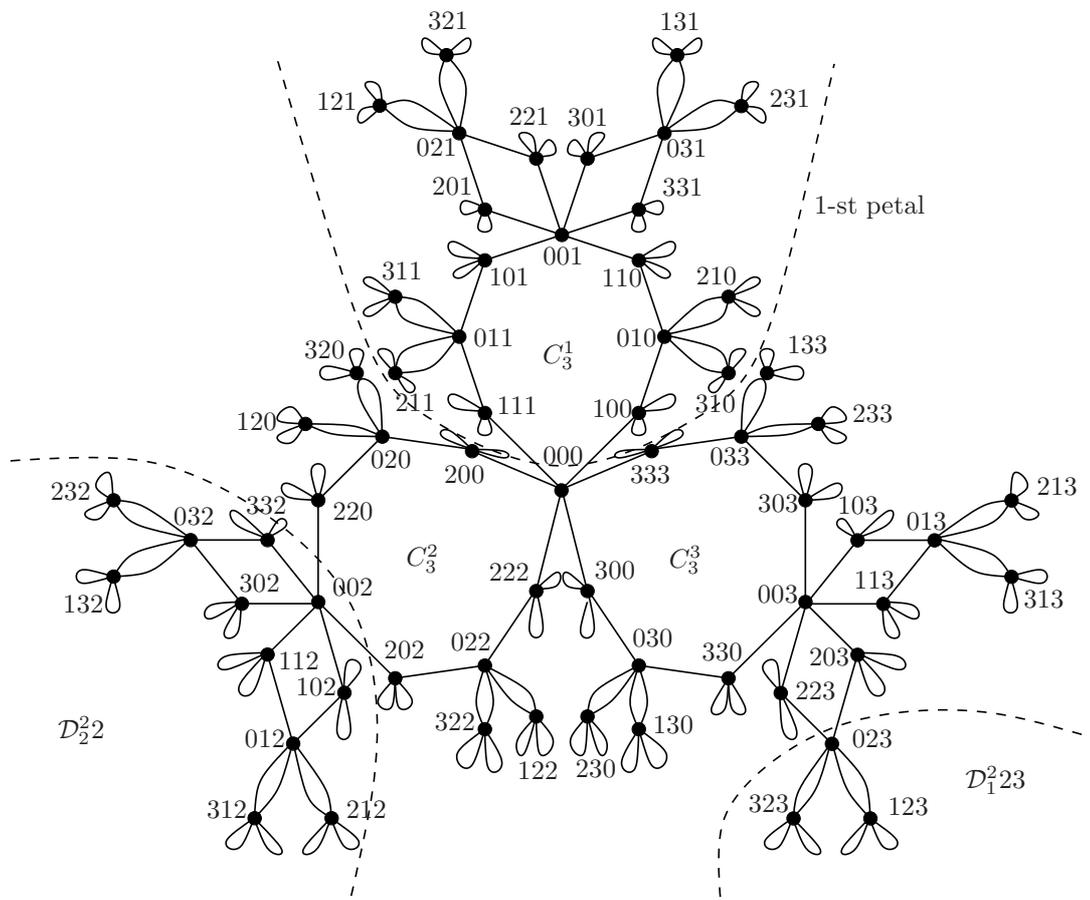}
\end{center}\caption{The Schreier graph $\Gamma^3_3$.} \label{alfn}
\end{figure}
\end{example}

Given a word $u\in X^n\setminus \{0^{n}\}$ we can write $u=0^ka_1u_1a_2u_2\ldots a_tu_t$, where $0\leq k\leq n-1$,  $a_i\in \{1,\ldots, p\}$, $a_i\neq a_{i+1}$  and $u_i\in \{0,a_i\}^{\ast}$. We call this writing the \emph{decomposition} of $u$.

\begin{lemma}\label{lemmaA}
Let $u\in X^n\setminus \{0^n\}$ and let $u=0^ka_1u_1a_2u_2\ldots a_tu_t$ be its decomposition. Then:
\begin{enumerate}
\item $m_u=t$
\item
$$
P_{1}^n(u)=C^{a_1}_{|u_1|+k+1}a_2u_2\ldots a_tu_t,  \quad P_{2}^n(u)=C^{a_2}_{|u_1|+|u_2|+k+2}a_3u_3\ldots a_tu_t , \quad \ldots
$$
$$
P_{i}^n(u)=C^{a_i}_{\sum_{\ell=1}^i|u_{\ell}|+k+i}a_{i+1}u_{i+1}\ldots a_tu_t,\quad \ldots, \quad  P_{m_u}^n(u)=C_n^{a_t}.
$$
\item
$$
\mathcal{L}_n^u=\{2^{|u_1|+k+1}, 2^{|u_1|+|u_2|+k+2}, \ldots,2^{\sum_{\ell=1}^i|u_{\ell}|+k+i} ,\ldots, 2^n\}.
$$
\end{enumerate}
\end{lemma}
\begin{proof}
We proceed by induction on the value of $t$ in the decomposition of $u$. \\
If $t=1$, then $u=0^ka_1u_1$ and such vertex belongs to $C_{|u|}^{a_1}$ and the claim is true.\\ Let $t=\ell+1$, so that $u=0^ka_1u_1a_2u_2\ldots a_{\ell}u_{\ell}a_{\ell+1}u_{\ell+1}$. Notice that the vertices $u$ and $v=0^{k+1+|u_1|}a_2u_2\ldots a_{\ell}u_{\ell}a_{\ell+1}u_{\ell+1}$ belong to the same cycle $C_{k+1+|u_1|}^{a_1}a_2u_2\ldots a_{\ell}u_{\ell}a_{\ell+1}u_{\ell+1}$ whose length is $2^{k+1+|u_1|}$. Notice that the index $t$ of the decomposition of $v$ equals $\ell$. By using the inductive hypothesis and the uniqueness of the path of cycles one can show the asserts.
\end{proof}


\begin{proposition}\label{propositionA}
Let $w=x_1x_2\ldots \in X^{\infty}\setminus \{0^{\infty}\}$ and consider the sequence of sets $ \{\mathcal{P}_n^{x_1\ldots x_n}\}_{n\geq 1}$. Then $|\mathcal{P}_n^{x_1\ldots x_n}|\leq |\mathcal{P}_{n+1}^{x_1\ldots x_nx_{n+1}}|$. Moreover $\lim _n |\mathcal{P}_n^{x_1\ldots x_n}|<\infty$ if and only if $w$ is cofinal to a word in $\{0,i\}^{\infty}$, for some $i\in X$.
\end{proposition}
\begin{proof}
Suppose that $w_n=x_1\ldots x_n$ ends with a suffix $i0^k$, for some $k\geq 0$ and $i\in\{1,\ldots, p\}$, so that it belongs to the $i$-th petal. Then, by using Lemma \ref{lemmaA}, passing from $\Gamma^p_n$ to $\Gamma^p_{n+1}$ we have two possible situations:
\begin{enumerate}
\item if $x_{n+1}\in\{0,i\}$, the index $t$ of the decomposition of $w_n$ and $w_{n+1}$ is the same.
\item  if $x_{n+1}\neq 0,i$ the index $t$ of the decomposition of $w_{n+1}$ increases by one with respect to that of $w_n$.
\end{enumerate}
The length of the path of cycles remains the same if and only if we add, after some prefix of $w$  ending with a suffix $i0^k$, only letters from the alphabet $\{0,i\}$, for some $i$.
In particular it follows that, in the second case, we have a nested path of cycles associated with the prefixes $w_n$ of $w$.
\end{proof}
Notice that the analogous statement clearly holds by substituting $ \mathcal{P}$ by $\mathcal{L}$.

\begin{remark}\label{remarkA}
Lemma \ref{lemmaA} and Proposition \ref{propositionA} imply that $\mathcal{P}_n^{x_1\ldots x_n}$ and $\mathcal{P}_{n+1}^{x_1\ldots x_nx_{n+1}}$ are such that either they have the same size (and in this case they differ just for the last cycle that has length $2^n$ in one case and $2^{n+1}$ in the other case) or the path $\mathcal{P}_{n+1}^{x_1\ldots x_nx_{n+1}}$ contains one cycle more than $\mathcal{P}_n^{x_1\ldots x_n}$ that is its subset.
In particular, the length of the path of cycles associated with $u=0^ka_1u_1a_2u_2\ldots a_tu_t$ is $t$. Any time we read a new letter $a_i$ the sequence increases by one.
\end{remark}
Remark \ref{remarkA} implies that one can define the path of cycles associated with $w=x_1x_2\ldots \in X^{\infty}$ as the limit of $\mathcal{P}_n^{x_1\ldots x_n}$. The same can be said for the sequence of the lengths. We denote them by $\mathcal{P}^w=\{P_1^w,P_2^w,\ldots \}$ and $\mathcal{L}^w=\{L_1^w,L_2^w,\ldots \}$, respectively. Moreover, we can also define the decomposition of an infinite word $u=0^ka_1u_1\ldots \in X^{\infty}$.

\subsection{From finite to infinite Schreier graphs}\label{sectioninfinite}
We start this section with the following result that is standard in this setting.

\begin{lemma}\label{lemmaCOF}
Let $w\in X^{\infty}$. If $w\in  Cof(0^{\infty})\cup\ldots\cup Cof(p^{\infty})$ then the orbit of $w$ under $\mathcal{G}_{S_p}$ coincides with $Cof(0^{\infty})\cup\ldots\cup Cof(p^{\infty})$. Otherwise the orbit of $w$ under $\mathcal{G}_{S_p}$ coincides with $Cof(w)$.
\end{lemma}
\begin{proof}
Notice that the only infinite paths in the generating automaton, that do not fall into the sink, are those labeled by $0^{\infty}|i^{\infty}$ starting at $e_i$, with $i\in\{1,\ldots, p\}$ (in particular, all the words $i^{\infty}$'s are in the orbit of $0^{\infty}$). This implies that the action of $\mathcal{G}_{S_p}$ changes infinitely many letters only on words of type $w=i^{\infty}$, with $i\in\{0,1,\ldots, p\}$. Therefore, if $w\in  Cof(0^{\infty})\cup\ldots\cup Cof(p^{\infty})$, its orbit is contained in $Cof(0^{\infty})\cup\ldots\cup Cof(p^{\infty})$; similarly, if $w\not\in Cof(0^{\infty})\cup\ldots\cup Cof(p^{\infty})$, one has that its orbit is contained in $Cof(w)$. In order to show the opposite inclusions, we use that $\mathcal{G}_{S_p}$ is fractal and spherically transitive (see \cite{articolo0}). In particular, given $u$ and $w$ cofinal, there exist prefixes $u_n, w_n$ of length $n$ such that $u=u_nv$ and $w=w_nv$. By transitivity, there exists $g\in \mathcal{G}_{S_p}$ such that $g(w_n)=u_n$. Let $g(w)=u_nv'$. By fractalness, there exists $g'\in \mathcal{G}_{S_p}$ such that $g'(u_nv')=u_nv$. Then $g'g(w)=u$, so that $u$ belongs to the orbit of $w$.
\end{proof}

The particular structure of the Schreier graphs allows to keep trace of the dynamic of an infinite word $u$.

\begin{lemma}\label{remarkB}
Let $u\in X^{\infty}$ with decomposition $u=0^ka_1u_1\ldots \in X^{\infty}$, then $P_{i}^u\cap P_{i+1}^u=\{0^{k+i+\sum_{j=1}^i|u_j|}a_{i+1}u_{i+1}\ldots \}$.
\end{lemma}
\begin{proof}
Take $n$ such that $n>k+i+1+\sum_{j=1}^{i+1}|u_j|$ and let $u_n$ be the prefix of $u$ of length $n$, such that $u=u_nu'$. Notice that if $P_{i}^{u_n}\cap P_{i+1}^{u_n}=\{w\}$ then $P_{i}^u\cap P_{i+1}^u=wu'$. This means that we can study the intersection of cycles in $\Gamma^p_n$ for $n$ large enough. A new cycle appears whenever we read a letter $a_{i+1}\neq a_i$. In this case $u_n$ becomes an element of $\mathcal{D}^{a_{i}}_{n}a_{i+1}$. In particular, the last two cycles are connected in $0^{k+i+\sum_{j=1}^i|u_j|}a_{i+1}$.
\end{proof}
Using the previous results we are ready to prove the following classification theorem. We recall that an $end$ for an infinite graph $\Gamma$
is an equivalence class of rays that remain in the same connected component whenever we remove a finite subgraph from $\Gamma$. An infinite graph $\Gamma$ is said to be \emph{$k$-ended} if it contains exactly $k$ ends. Equivalently, $\Gamma$ is $k$-ended if the supremum of the number of connected infinite components of $\Gamma$, when a finite subgraph is removed from $\Gamma$, equals $k$. For each $w\in X^\infty$, let us denote by $\Gamma^p_w$ the infinite Schreier graph of the group $\mathcal{G}_{S_p}$ containing the vertex $w$, that is, the graph describing the orbit of $w\in \partial T_{p+1}$ under the action of $\mathcal{G}_{S_p}$. Put $E_k=\{w\in X^{\infty}: \  \Gamma^p_w \textrm{ is $k$-ended}\}$.

Notice that, using the spherical transitivity of $\mathcal{G}_{S_p}$, one can show that any invariant measurable subset of $X^{\infty}$ must have measure $0$ or $1$ (see \cite{GNS}).

\begin{theorem}\label{fini}
Let $v\in X^\infty$. Then $\Gamma^p_v$ is either $2p$-ended, or $2$-ended, or $1$-ended. In particular:
\begin{enumerate}
\item $E_{2p}= Cof(0^{\infty}) \cup Cof(1^{\infty})\cup \cdots \cup Cof(p^{\infty})$ and consists of one orbit.
\item $E_2= \left(\bigcup_{i=1}^p \cup_{w\in \{0,i\}^{\infty}} Cof(w)\right) \setminus E_{2p} $ and consists of uncountably many orbits.
\item $E_1=X^{\infty}\setminus (E_{2p}\cup E_2)$  and consists of uncountably many orbits.
\end{enumerate}
Moreover $\nu(E_1)=1$.
\end{theorem}
\begin{proof}
\begin{enumerate}
\item The vertex $0^n$ belongs to $C_n^i$, for any $n\geq 1$ and for every $i\in \{1,\ldots, p\}$. When $n$ goes to infinity, the length of $C_n^i$ goes to infinity giving rise to $2$ rays that can be disconnected by removing the vertex $0^{\infty}$. The same can be said for the other cycles containing $0^n$, and this implies that $\Gamma^p_{0^{\infty}}$ is at least $2p$-ended. Any other vertex of $\Gamma^p_{0^{\infty}}$ belongs to some decoration $\mathcal{D}_k$, for some $k\in \mathbb{N}$, that is a finite graph attached to exactly one of the $2p$ rays described above. This implies that $\Gamma^p_{0^{\infty}}$ is $2p$-ended. Moreover, it follows from Lemma \ref{lemmaCOF} that $\Gamma^p_{0^{\infty}}=Cof(0^{\infty}) \cup Cof(1^{\infty})\cup\ldots \cup Cof(p^{\infty})$. This shows that $Cof(0^{\infty}) \cup Cof(1^{\infty})\cup\ldots \cup Cof(p^{\infty}) \subseteq E_{2p}$. The claim will follow from the remaining part of the proof.
\item Let $w=x_1x_2\ldots $ be cofinal to $u\in \{0,i\}^{\infty}\setminus (Cof(0^{\infty}) \cup Cof(i^{\infty}))$, for some $i\in\{1,\ldots,p\}$. By Proposition \ref{propositionA} the path of cycles associated with $w$ is finite. Moreover $d(w,u)<\infty$.  This implies that there exists $N\in \mathbb{N}$ such that $m_u=N$, and so $P_{m_u}(x_1\ldots x_n)=P_{N}(x_1\ldots x_n)=C_n^i$ for every $n$ large enough. The length of  $C_n^i$ is $2^n$ and goes to infinity. Hence $w$ belongs to a decoration attached at $u$ to an infinite double ray and so $\Gamma^p_w$ is $2$-ended. Finally, Lemma \ref{lemmaCOF} implies that each orbit coincides with a cofinality class. 
\item Any $w\in X^{\infty}\setminus \left(\bigcup_{i=1}^p \cup_{w\in \{0,i\}^{\infty}} Cof(w)\right)$ gives rise to an infinite path of cycles which is by construction unique. It follows that $\Gamma^p_w$ is $1$-ended. Also in this case, Lemma \ref{lemmaCOF} implies that each orbit coincides with a cofinality class.
\end{enumerate}
For the last claim, first observe that $E_{2p}$ is countable and so $\nu(E_{2p})=0$. In order to prove that $\nu(E_2)=0$, we notice that
$$
E_2\subset \bigcup_{i=1}^p \bigcup_{n\geq 0} X^n\{0,i\}^{\infty}.
$$
Let $i\in \{1,\ldots,p\}$ and let us show that $ \nu(\{0,i\}^{\infty})=0$. A direct computation gives
$$
\nu(\{0,i\}^{\infty})= 1-(p-1)\sum_{j=1}^{\infty}\frac{2^{j-1}}{(p+1)^j}=0.
$$
It follows that $\nu(E_2)\leq p \sum_{n=0}^{\infty}(p+1)^n \nu(\{0,i\}^{\infty})=0$.  Therefore
$$
1=\nu(E_1)+\nu(E_2)+\nu(E_{2p})=\nu(E_{1}).
$$
\end{proof}
In words, we can say that $E_2$ consists of infinite words containing, after any arbitrary finite prefix long enough, both the letters $0$ and $i$, for one fixed $i\in\{1,\ldots, p\}$, and only them. On the other hand, the set $E_1$ consists of infinite words containing, after any arbitrary finite prefix, at least two letters in $\{1,\ldots, p\}$.
\begin{remark}
Theorem \ref{fini} can be directly proven by using the techniques developed in \cite{BDN}.
\end{remark}

Now we pass to the study of isomorphism classes for the infinite Schreier graphs $\{\Gamma^p_w\}_{w\in X^{\infty}}$.

Let $w=x_1x_2\ldots\in X^{\infty}$. Recall that $(\Gamma^p_w,w)$ is the rooted graph obtained as limit of the finite rooted graphs $(\Gamma^p_n, x_1\ldots x_n)$ in the Gromov-Hausdorff topology. Once we get $(\Gamma^p_w,w)$ we forget the root and consider the infinite graph $\Gamma^p_w$. Given $u,v\in X^{\infty}$ we ask when $\Gamma^p_u$ and $\Gamma^p_v$ are isomorphic.        \\
Observe that the vertices belonging to $E_{2p}$ give rise to one isomorphism class, since they belong to the same orbit (the one containing $0^{\infty}$). Moreover, it is clear that graphs with different number of ends cannot be isomorphic.

We start with the following result.

\begin{lemma}\label{LemmaAbis}
Let $u,v\in E_{1}$ with $v\in \Gamma^p_u$. Then the sequences  $\mathcal{L}^u$ and $\mathcal{L}^v$ are compatible.
\end{lemma}
\begin{proof}
The sequences of cycles associated with $u$ and $v$ must eventually coincide. This exactly means that after some possibly different initial paths, the sequences must join. This implies that the sequences of the lengths of these cycles are compatible.
\end{proof}

\begin{lemma}\label{LemmaB}
Let $u,v\in E_1$ such that $\Gamma^p_u$ is isomorphic to $\Gamma^p_v$. Then the sequences $\mathcal{L}^u$ and $\mathcal{L}^v$ are compatible.
\end{lemma}
\begin{proof}
Notice that if $\Gamma^p_u$ is isomorphic to $\Gamma^p_v$, then there exists $w\in \Gamma^p_v$ and an isomorphism $\phi:\Gamma^p_u \rightarrow \Gamma^p_v$ such that $\phi(u)=w$. This implies that $\mathcal{L}^u=\mathcal{L}^w$. Since $w\in \Gamma^p_v$, then  Lemma \ref{LemmaAbis} implies that $\mathcal{L}^v$ and $\mathcal{L}^w$ are compatible. The claim follows.
\end{proof}

Whenever $w\in E_1$ there is also a sequence of vertices $\{w(n)\}_{n\in \mathbb{N}}$ defined by $w(0)=w$ and $\{w(i)\}=P_{i}^w\cap P_{i+1}^w$. If $w=0^ka_1u_1\ldots a_iu_i\ldots$ with $a_i\in \{1,\ldots, p\}$, $u_i\in\{0,a_i\}^{\ast}$ and $a_i\neq a_{i+1}$, then from Lemma \ref{remarkB} $w(i)=0^{k+i+\sum_{j=1}^i |u_j|}a_{i+1}u_{i+1}\ldots$. Moreover we  define the sequence of distances $\{d_n^w\}_{n\in\mathbb{N}}$ such that $d_i^{w}=d(w(i-1), w(i))$.

\begin{proposition}\label{propiso}
Let $u,v\in E_1$ such that $\{d_n^u\}_{n\in\mathbb{N}}$ and $\{d_n^v\}_{n\in\mathbb{N}}$ are compatible, then $\mathcal{L}^u$ and $\mathcal{L}^v$ are compatible.
\end{proposition}
\begin{proof}
Let $u=0^ka_1u_1\ldots a_iu_i\ldots$ with $a_i\in \{1,\ldots, p\}$, $u_i\in\{0,a_i\}^{\ast}$, $a_i\neq a_{i+1}$ and $v=0^mb_1v_1\ldots b_iv_i\ldots$ with $b_i\in \{1,\ldots, p\}$, $v_i\in\{0,b_i\}^{\ast}$, $b_i\neq b_{i+1}$ and suppose that $\mathcal{L}^u$ and $\mathcal{L}^v$ are not compatible. Then, for every $h,l\geq 0$, there exist infinitely many $n\in \mathbb{N}$ such that $L_{h+n}^u\neq L_{l+n}^v$. By Lemma \ref{lemmaA} this is equivalent to say
\begin{eqnarray}\label{distancereduction}
2^{\sum_{\ell=1}^{h+n}|u_{\ell}|+k+h+n}\neq 2^{\sum_{\ell=1}^{l+n}|v_{\ell}|+m+l+n}
\end{eqnarray}
for infinitely many $n\in\mathbb{N}$. Notice that, by virtue of Lemma \ref{remarkB}:
\begin{eqnarray*} d_{h+n+1}^{u}&=&d(u(h+n), u(h+n+1))\\
&=&d(0^{k+h+n+\sum_{j=1}^{h+n}|u_j|}a_{h+n+1}u_{h+n+1}\ldots, 0^{k+h+n+1+\sum_{j=1}^{h+n+1}|u_j|}a_{h+n+2}u_{h+n+2}\ldots)\\
&=&d(0^{k+h+n+\sum_{j=1}^{h+n}|u_j|}a_{h+n+1}u_{h+n+1}, 0^{k+h+n+\sum_{j=1}^{h+n+1}|u_j|}).
\end{eqnarray*}
In words, the distance $d^u_{h+n+1}$ can be computed within the finite Schreier graph $\Gamma^p_{k+h+n+\sum_{j=1}^{h+n+1}|u_j|}$.
The last distance relies to vertices belonging to the same cycle (the maximal cycle $C^{a_{h+n+1}}_{k+h+n+\sum_{j=1}^{h+n+1}|u_j|})$ and can be explicitly computed: suppose that $u_{h+n+1}=x_1\ldots x_{|u_{h+n+1}|}$, where $x_i\in\{0,a_{h+n+1}\}$. Put $t=k+h+n+\sum_{j=1}^{h+n}|u_j|$.  Then, by using the adding machine structure, one has
$$
d_{h+n+1}^{u}=\min\{2^{t}+\sum_{i: x_i\neq 0}2^{t+i}, 2^{t}+\sum_{i: x_i= 0}2^{t+i}\}.
$$
Analogously, if $s=m+l+n+\sum_{j=1}^{l+n}|v_j|$ and $v_{h+n+1}=y_1\ldots y_{|v_{l+n+1}|}$, where $y_i\in\{0,b_{l+n+1}\}$, one has
$$
d_{l+n+1}^{v}=\min\{2^{s}+\sum_{i: y_i\neq 0} 2^{s+i}, 2^{s}+\sum_{i: y_i= 0} 2^{s+i}\}.
$$
In any case $2^t$ is the smallest addend of $d_{h+n+1}^{u}$ and $2^s$ is the smallest addend of $d_{l+n+1}^{v}$. Since by Eq. \eqref{distancereduction} it must be $2^t\neq 2^s$, we get $d_{h+n+1}^{u}\neq d_{l+n+1}^{v}$ for infinitely many $n$. The claim follows.

\end{proof}

\begin{proposition}\label{prop11}
Let $w,v\in E_1$. 
Then $\Gamma^p_w$ is isomorphic to $\Gamma^p_v$ if and only if the sequences $\{d_n^w\}_{n\in\mathbb{N}}$ and $\{d_n^v\}_{n\in\mathbb{N}}$ are compatible.
\end{proposition}
\begin{proof}
Suppose that $\Gamma^p_w$ and $\Gamma^p_v$ are isomorphic. Then there exists $z\in \Gamma^p_v$ such that $(\Gamma^p_w,w)$ and $(\Gamma^p_v,z)$ are isomorphic as rooted graphs. Since the paths of cycles associated with $w$ and $z$ coincide, we have $d_n^w=d_n^z$ for every $n\geq 0$. Therefore $\{d_n^z\}_{n\in\mathbb{N}}$ and $\{d_n^v\}_{n\in\mathbb{N}}$ are compatible, because $z\in \Gamma^p_v$ and so the paths of cycles of $z$ and $v$ must join. The claim follows. \\
\indent Viceversa, first suppose that $d_n^w=d_n^v$ for every $n\in\mathbb{N}$. Then by adapting the proof of Proposition \ref{propiso} we deduce that $\mathcal{L}^w=\mathcal{L}^v$.  We want to define an isomorphism $\phi:\Gamma^p_w\to\Gamma^p_v$. First of all, put $\phi(w)=v$. Notice that $w(1)$ (resp. $v(1)$) is the only vertex of $P_1^w$ (resp. $P_1^v$) that is attached to a cycle isomorphic to $P_2^w$ (resp. $P_2^v$). Moreover, the cycles $P_2^w$ and $P_2^v$ are isomorphic by assumption. Being $d(w,w(1))=d(v,v(1))$, we can put $\phi(w(1))=v(1)$. By iterating the same argument for each $n$, we deduce that it must be $\phi(w(n))=v(n)$ for any $n\geq 0$. It follows that $\phi(P_n^w)=P_n^v$ for every $n$. Since such cycles have the same size, they have attached subgraphs that are isomorphic. This implies that $\phi$ can be extended to an isomorphism between $\Gamma^p_w$ and $\Gamma^p_v$.\\
\indent If the sequences $\{d_n^w\}_{n\in\mathbb{N}}$ and $\{d_n^v\}_{n\in\mathbb{N}}$ are compatible, then there exist $i,j$ such that $d_{i+n}^w=d_{j+n}^v$ for every $n\in\mathbb{N}$. Define $\phi(w(i+n))=v(j+n)$ for each $n$. Notice that $\Gamma^p_w\setminus\{w(i)\}$ contains one infinite connected component which is isomorphic to the only infinite one of  $\Gamma^p_v\setminus\{v(j)\}$. The remaining parts of the two graphs $\Gamma^p_w$ and $\Gamma^p_v$ are finite subgraphs attached to the isomorphic cycles $P_{i+1}^w$ and $P_{j+1}^v$ and so they are isomorphic. This gives an isomorphism between $\Gamma^p_w$ and $\Gamma^p_v$.
\end{proof}

Pay attention to the fact that there are infinite sequences $u,v\in X^{\infty}$ such that $\mathcal{L}^u=\mathcal{L}^v$, but $\Gamma^p_u$ and $\Gamma^p_v$ are not isomorphic.
\begin{example}
Consider the vertices $u=(1002)^{\infty}$ and $v=(1012)^{\infty}$. The path of cycles associated with $u=(1002)^{\infty}$ and $v=(1012)^{\infty}$ is the same. However, by using Proposition \ref{prop11}, one can check that there is no isomorphism between the graphs $\Gamma^p_u$ and $\Gamma^p_v$, since $u(n)$ and $v(n)$ belong to two cycles of the same length, but $d(u(n),u(n+1)) \neq d(v(n),v(n+1))$ for each $n$.
\end{example}

For every $i,j \in \{1,\ldots, p\}$, let us define the map $\phi_{i,j}:X\to X$ as:
$$
\phi_{i,j}(k) = \left\{
                  \begin{array}{ll}
                    k & \hbox{if } k\neq i \\
                    j & \hbox{if } k=i.
                  \end{array}
                \right.
$$
In particular, the map $\phi_{i,j}$ fixes $0$, for any $i,j$; moreover, $\phi_{i,i}$ is the identity map. Given $w=x_1x_2\ldots\in X^{\ast}\cup X^{\infty}$ we define $\phi_{i,j}(w)=\phi_{i,j}(x_1)\phi_{i,j}(x_2)\ldots$.

Given $u\in \{0,i\}^{\ast}\cup \{0,i\}^{\infty}$, we denote by $u'$ the word obtained from $u$ by switching $0$ to $i$ and viceversa. For a given element $u\in X^{\infty}$ cofinal to a word in $\{0,i\}^{\infty}$, having the form $u_1iu_2$, where $iu_2$ is the maximal suffix of $u$ in $\{0,i\}^{\infty}$, we put $\overline{u}=u_1iu_2'$.

\begin{example}\label{esempioserpe}
Let $w=0^ka_1u_1a_2u_2a_3u_3\ldots$, with $a_i\in \{1,2,\ldots, p\}$, $u_i\in \{0,a_i\}^\ast$ and $a_{i+1}\neq a_i$ as usual. We have:
$$
d^w_1=d(w,w(1)); \ \ d^w_2=d(w(1),w(2)); \ \ d^w_3=d(w(2),w(3));\  \ d^w_4=d(w(3),w(4)),
$$
with
\begin{eqnarray*}
w(1)= 0^{k+1+|u_1|}a_2u_2\ldots &;& \quad w(2)= 0^{k+2+|u_1|+|u_2|}a_3u_3\ldots,  \\
w(3)= 0^{k+3+|u_1|+|u_2|+|u_3|}a_4u_4\ldots &;& \quad  w(4)= 0^{k+4+|u_1|+|u_2|+|u_3|+|u_4|}a_5u_5\ldots.
\end{eqnarray*}
Now consider the vertex $v=0^ka_1u_1'a_2u_2a_3u_3\ldots$, with $u_1'$ obtained from $u_1$ by switching $0$ to $a_1$ in $u_1$ and viceversa (see Fig. \ref{figserpentone}).
Notice that $w(n) \equiv v(n)$, for each $n\geq 1$, since $w$ and $v$ belong to the same cycle. It follows that $\mathcal{P}^w=\mathcal{P}^{v}$, so that $\mathcal{L}^w = \mathcal{L}^{v}$.\\
Finally, a comparison between the decompositions of $w$ and $v$ ensures that $d^w_1=d(w,w(1)) = d^{v}_1=d(v,v(1))$, so that we also have $d^w_n= d^{v}_n$ for each $n\geq 1$.
\begin{figure}[h]
\begin{center}
\psfrag{w}{$w\equiv w(0)$}\psfrag{w'}{$v$}\psfrag{P1}{$P^w_1$}\psfrag{P2}{$P^w_2$}\psfrag{P3}{$P^w_3$}\psfrag{P4}{$P^w_4$}\psfrag{P5}{$P^w_5$}
\scriptsize
\psfrag{w1}{$w(1)$}\psfrag{w2}{$w(2)$}\psfrag{w3}{$w(3)$}\psfrag{w4}{$w(4)$}
\includegraphics[width=0.5\textwidth]{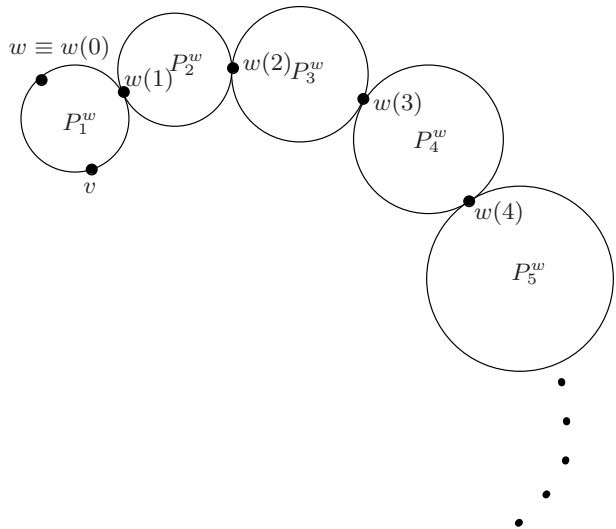}
\end{center}\caption{The path of cycles of the word $w$ of Example \ref{esempioserpe}.} \label{figserpentone}
\end{figure}
\end{example}

\begin{theorem}\label{teo_iso}
\begin{enumerate}
\item There is one isomorphism class of $2p$-ended graphs consisting of the graph $\Gamma^p_{0^{\infty}}$.
\item Let $u,v\in E_2$, Then $\Gamma^p_u$ is isomorphic to $\Gamma^p_v$ if and only if either $v\sim \phi_{i,j}(u)$ or $
v\sim \overline{\phi_{i,j}(u)}$ for some $i,j\in \{1,\ldots, p\}$. In particular there are uncountably many isomorphism classes of $2$-ended graphs, each consisting of $2p$ graphs.
\item Let $u,v\in E_1$ with $u=0^ka_1u_1a_2u_2\ldots a_iu_i\ldots$, $a_i\in\{1,\ldots, p\}$, $a_i\neq a_{i+1}$ and $u_i\in\{0,a_i\}^{\ast}$. Then $\Gamma^p_u$ is isomorphic to $\Gamma^p_v$ if and only if
$v\sim 0^kb_1v_1b_2v_2\ldots b_iv_i\ldots$ where, for any $i\in \mathbb{N}$,  $b_i\in\{1,\ldots, p\}$, $|v_i|=|u_i|$, $b_i\neq b_{i+1}$ and either $v_i=\phi_{a_i,b_i}(u_i)$ or $v_i=\phi_{a_i,b_i}(u_i)'$. In particular, there are uncountably many isomorphism classes of $1$-ended graphs, each consisting of uncountably many graphs.
\end{enumerate}
\end{theorem}
\begin{proof}
\begin{enumerate}
\item The first statement is clear by Theorem \ref{fini}.
\item Let $u\in E_2$. Then by Theorem \ref{fini} there exists $i\in\{1,\ldots, p\}$ such that $u=u_1iu_2$, where $iu_2$ is the maximal suffix of $u$ in $\{0,i\}^{\infty}$. This implies that $u$ belongs to a decoration isomorphic to $\mathcal{D}_{|u_1|}$ and attached at $0^{|u_1|} iu_2$ to the corresponding infinite double ray. Notice that, due to the decomposition of $u$, the graphs $\Gamma^p_u$ and $\Gamma^p_{\phi_{i,j}(u)}$ are isomorphic. In particular, such isomorphism maps the decoration attached at  $0^{|u_1|}iu_2$ to the isomorphic one attached at $0^{|u_1|}\phi_{i,j}(iu_2)$. This isomorphism maps the vertices of the infinite double ray $e_i^{t}(0^{|u_1|}iu_2)$ to $e_j^{t}(0^{|u_1|}\phi_{i,j}(iu_2))$, for each $t\in \mathbb{Z}$. Similarly, the graphs $\Gamma^p_u$ and $\Gamma^p_{\overline{\phi_{i,j}(u)}}$ are isomorphic, and the isomorphism maps $0^{|u_1|}iu_2$ to $0^{|u_1|} \overline{\phi_{i,j}(iu_2)}$. It follows that the vertices of the infinite double ray $e_i^{t}(0^{|u_1|}iu_2)$ are mapped to $e_j^{-t}(0^{|u_1|}\overline{\phi_{i,j}(iu_2)})$, for each $t\in \mathbb{Z}$, as one can deduce from  Proposition \ref{prop1}. Finally, by using Lemma \ref{lemmaCOF}, it is easy to show that, if $v$ is either cofinal to $\phi_{i,j}(u)$ or cofinal to $\overline{\phi_{i,j}(u)}$, then $\Gamma^p_u$ is isomorphic to $\Gamma^p_v$.\\
    \indent Viceversa, suppose $u=x_1x_2\ldots=u_1iu_2$ and $v=y_1y_2\ldots=v_1jv_2$, where $iu_2$ is the maximal suffix of $u$ in $\{0,i\}^{\infty}$ and  $iv_2$ is the maximal suffix of $v$ in $\{0,j\}^{\infty}$. Assume that $\Gamma^p_u$ and $\Gamma^p_v$ are isomorphic through $\phi$ in such a way that $\phi(u)=v$. Then $\phi$ must induce an isomorphism of the finite rooted graphs $(\Gamma^p_n,x_1\ldots x_n)$ and $(\Gamma^p_n,y_1\ldots y_n)$ for each $n$. Take $n$ large enough so that $n>|u_1|+1$. By Proposition \ref{prop1}, the graphs $(\Gamma^p_n,x_1\ldots x_n)$ and $(\Gamma^p_n,y_1\ldots y_n)$ are isomorphic if and only if  either $jy_{|u_1|+2}\ldots y_{n}=\phi_{i,j}(ix_{|u_1|+2}\ldots x_{n})$ or $jy_{|u_1|+2}\ldots y_{n}=j\phi_{i,j}(x_{|u_1|+2}\ldots x_{n})'$. The claim follows.
\item First let $v= 0^kb_1v_1b_2v_2\ldots b_iv_i\ldots$ where for any $i\in \mathbb{N}$,  $b_i\in\{1,\ldots, p\}$, $|v_i|=|u_i|$, $b_i\neq b_{i+1}$ and either $v_i=\phi_{a_i,b_i}(u_i)$ or $v_i=\phi_{a_i,b_i}(u_i)'$. We claim that $d_n^v=d_n^u$ for any $n\in\mathbb{N}$.  Recall that
\begin{eqnarray*}
d_{n}^{v}&=&d(v(n-1), v(n))\\
&=&d(0^{k+\sum_{j=1}^{n-1}|v_j|}b_{n}v_{n}\ldots, 0^{k+\sum_{j=1}^{n}|v_j|}b_{n+1}v_{n+1}\ldots)\\
&=&d(0^{k+\sum_{j=1}^{n-1}|v_j|}b_{n}v_{n},0^{k+\sum_{j=1}^{n}|v_j|}).
\end{eqnarray*}
From the proof of Proposition \ref{propiso} it follows that the value of $d_{n}^{v}$ only depends on the position of the 0's and $b_n'$s in $v_n$ and so it is independent from the specific $b_n$. By assumption $k+\sum_{j=1}^{n}|v_j|=k+\sum_{j=1}^{n}|u_j|$ and by Proposition \ref{prop1} the vertex $v$ satisfies $d_{n}^{v}=d_{n}^{u}$ if $v_i=\phi_{a_i,b_i}(u_i)$ or $v_i=\phi_{a_i,b_i}(u_i)'$. This implies that $\Gamma^p_v$ and $\Gamma^p_u$ are isomorphic (as rooted graphs). Finally, Lemma \ref{lemmaCOF} implies that if $v'$ is cofinal to $v$, then $\Gamma^p_v = \Gamma^p_{v'}$ and so $\Gamma^p_{v'}$ is isomorphic to $\Gamma^p_{u}$.\\
Viceversa suppose $v= 0^mb_1v_1b_2v_2\ldots b_iv_i\ldots$, with $b_i\in\{1,\ldots, p\}$, $b_i\neq b_{i+1}$ and $v_i\in\{0,b_i\}^{\ast}$. First notice that Lemma \ref{LemmaB} implies that if $u,v\in E_1$ are such that $\Gamma^p_u$ is isomorphic to $\Gamma^p_v$, then the sequences $\mathcal{L}^u$ and $\mathcal{L}^v$ must be compatible. By Lemma \ref{lemmaA}, this is equivalent to say that there exist $l,h$ such that
\begin{eqnarray}\label{arturo}
|v_{l+i}|=|u_{h+i}| \ \forall \ i\geq 1     \ \mbox{ and }  \ m+l+\sum_{j=1}^{l}|v_j|=k+h+\sum_{j=1}^{h}|u_j|.
\end{eqnarray}
Moreover, Proposition \ref{prop11} implies that also the sequences $\{d_n^u\}_{n\in\mathbb{N}}$ and $\{d_n^v\}_{n\in\mathbb{N}}$ must be compatible. As before, one has
\begin{eqnarray*}
d_{l+i}^{v}&=&d(v(l+i-1), v(l+i))\\
&=&d(0^{k+l+i-1+\sum_{j=1}^{l+i-1}|v_j|}b_{l+i}v_{l+i}\ldots, 0^{k+l+i+\sum_{j=1}^{l+i}|v_j|}b_{l+i+1}v_{l+i+1}\ldots)\\
&=&d(0^{k+l+i-1+\sum_{j=1}^{l+i-1}|v_j|}b_{l+i}v_{l+i},0^{k+l+i+\sum_{j=1}^{l+i}|v_j|}).
\end{eqnarray*}
Then, using Eq. \eqref{arturo}, one can check that this quantity equals $d_{h+i}^{u}$ if and only if, for any $i \geq 1$, there exists $b_i\in \{1,\ldots, p\}$ such that either $v_{l+i+1}= \phi_{a_i,b_i}(u_{h+i+1})$ or $v_{l+i+1}= \phi_{a_i,b_i}(u_{h+i+1})'$. The claim follows.
\end{enumerate}
\end{proof}
\noindent For a given $w\in X^\infty$, put
$$
I_w= \{z\in X^\infty : \Gamma^p_z \mbox{ and } \Gamma^p_w \mbox{ are isomorphic}\}.
$$
\begin{corollary} \label{finalcorozero}
For every $w\in X^\infty$, one has $\nu(I_w)=0$.
\end{corollary}
\begin{proof}
Notice that, by virtue of Theorem \ref{fini}, if $w\in E_{2p}$ or $w\in E_2$, then $\nu(I_w)=0$. Therefore, we can restrict our attention to the case $w\in E_1$.
Let $w=0^ka_1u_1\ldots$ and let us consider its decomposition. By Theorem \ref{teo_iso}, we have $I_w = \bigcup_{u\in E_w'} Cof(u)$, where
$$
E_w'= \{z\in X^\infty : (\Gamma^p_z,z) \mbox{ and } (\Gamma^p_w,w) \mbox{ are isomorphic as rooted graphs}\}.
$$
In particular, the claim follows if we prove that $\nu(E'_w)=0$. By Claim (3) of Theorem \ref{teo_iso}, in order to measure $E_w'$, we first must remove from $X^{\infty}$ all subsets of type $0^tmX^{\infty}$, with $t\leq k-1$ and $m\neq 0$. Then in $0^kX^{\infty}$ we remove the subset $0^{k+1}X^{\infty}$. After that consider the subsets $0^kaX^{\infty}$,  with $a\in\{1,\ldots, p\}$. In each of these $p$ subsets of $X^\infty$, all subsets of type $0^kavX^{\infty}$ must be removed, for any $v$ of length $u_1$, except for the two $v$'s giving rise to isomorphism. After that, for each of the remaining ray, we proceed as before, according to the decomposition of $w$: the subsets of type $0^kavbX^{\infty}$ with $b\in\{0,a\}$ must be removed. By iterating this argument, a direct computation gives:
\begin{eqnarray*}
 \nu(E_w')&=& 1-p\sum_{j=1}^{k}\frac{1}{(p+1)^j}-\frac{1}{(p+1)^{k+1}}\\
&-& p\frac{(p+1)^{|u_1|}-2}{(p+1)^{|u_1|+k+1}}-2\frac{2p}{(p+1)^{|u_1|+k+2}}\\
&-& 2p(p-1)\frac{(p+1)^{|u_2|}-2}{(p+1)^{|u_1|+|u_2|+k+2}}- 2\frac{4p(p-1)}{(p+1)^{|u_1|+|u_2|+k+3}}\\
&-&  4p(p-1)^2\frac{(p+1)^{|u_3|}-2}{(p+1)^{|u_1|+|u_2|+|u_3|+k+3}} -\cdots
\end{eqnarray*}
We can rearrange the sum as follows
\begin{eqnarray*}
 \nu(E_w')&=& 1-p\sum_{j=1}^{k}\frac{1}{(p+1)^j}-\frac{1}{(p+1)^{k+1}}-\frac{p}{(p+1)^{k+1}}+\\
&-& \sum_{m=1}^{\infty}\left[\frac{2^mp(p-1)^{m-1}}{(p+1)^{|u_1|+\cdots +|u_m|+k+m}}- \frac{2^{m+1}p(p-1)^{m-1}}{(p+1)^{|u_1|+\cdots +|u_m|+k+m+1}}-\frac{2^mp(p-1)^{m}}{(p+1)^{|u_1|+\cdots +|u_m|+k+m+1}}\right]\\
&=&0+ \sum_{m=1}^{\infty}\frac{2^mp(p-1)^{m-1}}{(p+1)^{|u_1|+\cdots +|u_m|+k+m+1}}[(p+1)-2-p+1] =0.
\end{eqnarray*}
\end{proof}

\end{document}